\documentclass[11pt]{article}
\usepackage{amscd}
\usepackage{amsmath}
\usepackage{latexsym}
\usepackage{amsfonts}
\usepackage{amssymb}
\usepackage{amsthm}
\usepackage{cite}
\usepackage{verbatim}
\usepackage{enumerate}

\theoremstyle{plain}
\theoremstyle{definition}\newtheorem{theorem}{Theorem}[section]
\theoremstyle{plain}\newtheorem{lemma}[theorem]{Lemma}
\theoremstyle{plain}
\theoremstyle{plain}\newtheorem{proposition}[theorem]{Proposition}
\theoremstyle{remark}

\newcommand{\be}{\begin{equation}}
\newcommand{\ee}{\end{equation}}
 \newcommand{\ba}{\begin{aligned}}
 \newcommand{\ea}{\end{aligned}}

  \newcommand{\ben}{\begin{enumerate}}
   \newcommand{\een}{\end{enumerate}}

\newcommand{\Rmnum}[1]{\expandafter\@slowromancap\romannumeral #1@}

\allowdisplaybreaks
\begin{document}
\title{Stability of a superposition of shock waves with contact discontinuities for the Jin-Xin relaxation system}

\author{Hualin Zheng\\
\\
        \small{Mathematical Sciences Center, Tsinghua University,Beijing, 100084, P. R. China.}\\
        \small{Email:~~zhenghl12@mails.tsinghua.edu.cn}}

\date{}
\maketitle
\begin{abstract}
   In this paper, we consider the large time asymptotic nonlinear stability of a superposition
of shock waves with contact discontinuities for the one dimensional Jin-Xin relaxation system
with small initial perturbations, provided that the strengths of waves are
small with the same order. The results are obtained by elementary weighted energy estimates
based on the underlying wave structure and an estimate on the heat equation.
\end{abstract}

\section{Introduction}
\quad ~
Consider a semilinear hyperbolic relaxation model
\begin{equation}\label{lp1.1-1}
\begin{split}
 u_t + v_x = 0   \ \ {\rm and} \ \   v_t + a^2 u_x = \frac{1}{\varepsilon}\big[f(u) - v\big], \  \  x\in\mathbb{R},  \ \  t>0, \\
\end{split}
\end{equation}
where $u,v\in \mathbb{R}^n$ are unknown functions,  $f(u)\in \mathbb{R}^n$ is a given smooth function,
$\varepsilon>0$ is a small constant representing the rate of relaxation, $a>0$ is a given constant.
This model was introduced in \cite{jx} for numerical purposes and named as the Jin-Xin relaxation
system from the names of the authors. For a general introduction and survey to relaxation schemes, see \cite{natalini,w}.
As $\varepsilon$ tends to zero, the Jin-Xin relaxation system is expected to approach its equilibrium system:
\begin{equation}\label{lp1.1-4}
v=f(u) \ \  {\rm and } \ \ u_t + f(u)_x =0,
\quad   x\in \mathbb{R}, ~ t>0.\\
\end{equation}
Indeed, the passage from  \eqref{lp1.1-1} to \eqref{lp1.1-4} has been intensely studied, see
\cite{bianchini,cll,serre,xuwenqing}
for instance. The purpose of this work is to study the large time asymptotic behavior toward a superposition of relaxation shock waves and contact waves of solutions to the Cauchy problem of \eqref{lp1.1-1} with the  initial data
\begin{equation}\label{lp1.1-3}
(u,v)(x,t=0)=(u_0,v_0)(x), \ \ x\in \mathbb{R}; \ \  \lim_{x\rightarrow \pm \infty}(u_0, v_0)(x) = \big(u_\pm, f(u_\pm) \big).
\end{equation}

It is well known that solutions to the equilibrium system \eqref{lp1.1-4}
contain three basic wave patterns: shock waves, rarefaction waves
and contact discontinuities. These dilation invariant solutions and
their linear superpositions in the increasing order of
characteristic speed, called Riemann solutions, govern both the
local and large time asymptotic behavior of general solutions to
system \eqref{lp1.1-4} (\emph{cf.}\cite{L2}). Since the equilibrium  system
is an idealization when the relaxation effects are
neglected, thus it is of great importance to study the large time
asymptotic behavior of the solutions to the relaxation system \eqref{lp1.1-1} toward the relaxation versions of these basic waves: relaxation shock waves, relaxation rarefaction waves and relaxation contact
waves. In the study of the large time asymptotic
stability toward these relaxation waves, important progress have been made by various researchers in one or multiple space dimensions for various types of relaxation systems, mainly on  a single type of relaxation waves; one may refer to
\cite{hl,hp3,lhl1,lhl2,lt,l,ltxzp,mn,mz,my,ltlhl,zengyn,zhhj,zhcj,wangweicheng1,ywa,liu-yang-yu-zhao,yu}
and \cite{hpw,huang-xin-yang} and references therein for nonlinear relaxation waves (i.e., shock waves and rarefaction waves) and linear relaxation waves (i.e., contact waves), respectively. Among these works, the stability of a single relaxation shock wave and of a single relaxation contact wave for the Jin-Xin model were shown in \cite{lhl1} and \cite{hpw}, separately. However, it is worthy to point out that even though the stability of each wave pattern for relaxation models is now well understood, the stability of a wave pattern to Riemann solutions is still not known. In the present work, we will pursue this issue in the case that the Riemann solution contains shock waves and contact discontinuities.

Notice that system \eqref{lp1.1-1} can be rewritten, by differentiating the second equation in \eqref{lp1.1-1} with respect to $x$ and using the
first, as
\begin{equation}\label{eqvjx}
u_t + f(u)_x = a^2 \varepsilon u_{xx} - \varepsilon u_{tt},
\quad   x\in \mathbb{R}, ~ t>0.\\
\end{equation}
It was proved in \cite{hpw} that $u_{tt}$ in the equation above could be viewed as a higher-order perturbation term in the study  of the large time asymptotic
stability toward a single relaxation contact wave. Motivated by this, we may treat equation \eqref{eqvjx} as a perturbation of viscous hyperbolic conservation laws:
\begin{equation}\label{viscous}
u_t + f(u)_x = a^2 \varepsilon u_{xx},
 \ \   x\in \mathbb{R}, ~ t>0.
\end{equation}
To the best of our knowledge, there exist only two works discussing the stability of Riemann solutions consisting of different types of elementary waves for viscous conservation laws and their variations. One is for the compressible Navier-Stokes equation when the Riemann solution is a superposition of rarefaction waves and a contact discontinuity
(\emph{cf.} \cite{huang-li-matsumura}). The other is for viscous hyperbolic conservation laws with the Riemann solution containing shock waves and contact discontinuities
(\emph{cf.} \cite{z}). In the spirit of the latter work, we will show that the superposition of relaxation shock waves with contact waves is asymptotically stable for solutions to the initial value problem \eqref{lp1.1-1} and \eqref{lp1.1-3} with small initial perturbations, under the subcharacteristic condition \eqref{scc}. It should be pointed out the following structural condition:
$$\nabla l_p \cdot r_p = 0 \ \ {\rm if} \ \ p\textrm{-th characteristic field is linearly degenerate}, $$
required in \cite{hpw} to show the stability of a single contact wave, can be removed in the present work by choosing a slightly different variable to perform the higher-order estimate.

We comment on some of the main difficulties and techniques involved in our analysis. Since the initial perturbation is generic and there are $n$ different characteristic fields, wave interactions do occur; different wave patterns possess different properties in terms of monotonicity and decay rate, such that the $L^2$-estimate of the integrated error equation is complete for a single relaxation shock wave (\emph{cf.}\cite{lhl1}), but not for a single relaxation contact wave (\emph{cf.}\cite{hpw}); comparing with viscous conservation laws \eqref{viscous}, which is a parabolic system, \eqref{eqvjx} is a different type of equation, wave equation, new features occurs.
To overcome these difficulties and then obtain a complete $L^2$-estimate of the integrated error equations, we first diagonalize the integrated error equation and carry out some weighted energy estimates as in \cite{z}, based on the idea that the $i$-th relaxation waves dominates the $i$-th characteristic zone as time is large and the wave interactions decay exponentially fast with respect to the spatial and temporal variables for large time. Then we get the desired $L^2$-estimate \eqref{loemissp} except a term due to the constant characteristic speed along the relaxation contact wave. To deal with this term, we use some properties of heat kernel (Lemma \ref{gh}) and the structure of system \eqref{eqvjx} by verifying that the corresponding $u_{tt}$ term in the diagonalized equations for the linearly degenerate characteristic family is a higher-order term.
This verification is highly nontrivial since a wave equation could be viewed as a parabolic system under this verification. In this paper, we can verify it when the equilibrium system admits only one linearly degenerate characteristic family.

The rest of this paper is organized as follows. In Section 2, some notations, assumptions, and the main result are stated. Section 3 is devoted to collecting some lemmas concerning the properties of heat kernel, relaxation shock waves and contact waves. The main result, Theorem \ref{mainthm}, is proven in section 4.

\section{Mathematical setting and main results}
\quad ~
Before stating the main result, we introduce some preliminary notations and give some background materials. In this paper, the conservation law
\begin{equation}\label{newlp1.1-4}
u_t + f(u)_x =0,
\quad   x\in \mathbb{R}, ~ t>0,\\
\end{equation}
is assumed to be {\it strictly hyperbolic}, that is, the Jacobian $  f'(u)$ has $n$ distinct real eigenvalues:
\begin{equation*}\label{lp1.1-5}
\lambda_1(u) < \lambda_2(u)< \cdots <\lambda_n(u).
\end{equation*}
Denote left and right eigenvectors of the matrix $ f'(u)$ corresponding to the eigenvalue $\lambda_i$ $(i=1,\cdots,n)$ by $l_i(u)$ and $r_i(u)$, respectively. Define the matrices $L(u)$, $R(u)$ and  $\Lambda (u)$ by
\begin{equation*}
L := (l_1 ,  \cdots, l_n )^T,  \ \
R  := (r_1 ,  \cdots, r_n ), \ \  \Lambda  : = \mbox{diag} (\lambda_1,\cdots, \lambda_n ),
\end{equation*}
where and thereafter $(\cdots)^T$ denotes the transpose. Then, it holds that \begin{equation*}\label{lp1.1-6}
L(u)   f'(u) R(u) =\Lambda(u)  \ \  {\rm and} \ \  L(u) R(u) = \mathbb{I} : = \mbox{identity matrix}.
\end{equation*}
The $i$-th characteristic field
is called genuinely nonlinear (or linearly degenerate), if $\nabla
\lambda_i(u)\cdot r_i (u)\neq 0$ (or $\nabla \lambda_i(u)\cdot r_i
(u)= 0$) for all $u$ under consideration. An $i$-shock wave (or $i$-contact
discontinuity), denoted by a triple $(u_l,u_r,s)$, is a piecewise
constant weak solution to \eqref{newlp1.1-4} (\emph{cf.}\cite{lax,smoller}), such that
\begin{equation*}\left\{\begin{array}{l}
f(u_l)-f(u_r)=s(u_l-u_r),\\
\lambda_i (u_r)<s<\lambda_i(u_l) \quad(\textrm{or}\;
s=\lambda_i(u_r)=\lambda_i(u_l)).
\end{array} \right. \end{equation*}
Suppose that the Riemann solution to \eqref{newlp1.1-4} with the following initial data
\begin{equation}\label{lp1.1-10}
u(x,0)=\left\{\begin{array}{l}u_{-}\,,\qquad x<0,
\\u_{+}\,, \qquad x>0,\end{array}\right.
\end{equation}
consists of $(n+1)$ constant states, denoted by
$u_1=u_-,\;u_2,\;\cdots,\;u_{n},\;u_{n+1}=u_+$, separated by a contact discontinuity   and  $n-1$ shock waves with the wave speeds $s_1,\cdots, s_n$. Let $p\in[1,n]$ and $(u_p, u_{p+1}, s_p)$ be the contact discontinuity.
Denote the wave strengths by
$$\delta_i\equiv|u_{i+1}-u_i|, \quad i=1,\cdots,n;\quad \delta\equiv\underset{i=1,\cdots,n}\min \{\delta_i\}.$$
When $|u_+-u_-|$ is small, there exists a positive constant $C_1$,
depending only on the flux function  and the far field data, such that
$$\delta_1+\delta_2+ \cdots+ \delta_n \leqslant C_1|u_+-u_-|.$$
The strengths of these waves are said to be \lq\lq small with the
same order" if
\begin{equation}\label{strength}
\begin{split}
\delta_1+\delta_2+ \cdots+ \delta_n \leqslant C_2\delta \quad \textrm{as}
\;\;\delta_1+\delta_2+ \cdots+ \delta_n\to 0,
\end{split}
\end{equation}
for some positive constant $C_2$.

Without loss of generality, we assume $\varepsilon =1$ and the speed of the contact discontinuity, $s_p=0$. To ensure the dissipative nature of system \eqref{lp1.1-1}, we impose the sub-characteristic condition \cite{cll,jx,l,w} as follows
\begin{equation}\label{scc}
|\lambda_i(u)|< a, \  \ i=1,\cdots,n,
\end{equation}
for all $u$ under consideration. Next, we construct the $i$-th relaxation shock wave (or $p$-th contact wave)
for \eqref{lp1.1-1} toward the given $i$-shock wave (or $p$-contact
discontinuity) for \eqref{newlp1.1-4}.
When $i\neq p$, the $i$-th relaxation
shock wave $u^i(x,t)$ is defined as
$$u^i(x,t)\equiv\varphi^i(\xi_i)=\varphi^i(x-s_it).$$
Here $\varphi^i(\xi_i)$ is the smooth traveling solution of
\eqref{lp1.1-1}, satisfying
\begin{equation}\label{shequ}
\left\{\begin{array}{l}
\big[f'(\varphi^i)-s_i\mathbb{I}\big](\varphi^i)'=\big(a^2-s_i^2\big)(\varphi^i)'',\\
\varphi^i(-\infty)=u_i,\ \ \varphi^i(+\infty)=u_{i+1},
\end{array} \right.\end{equation}
where $'=d/(d\xi_i)$. While for $i=p$, set
\begin{equation*}\label{}
C_p(u_p)=\left\{\;u\;\Big|\;u=u(\rho) ,\; \frac{du}{d\rho}=r_p\big(u(\rho)\big)
,\;u(\rho=\rho_-)=u_p\;\right\}.
\end{equation*}
Thus the $p$-th contact wave curve through $u_p$, $C_p(u_p)$, is the integral curve associated with the vector field $r_p(u)$ in the state space with the nonsingular parameter $\rho$. The parameter
$\rho$ is chosen to satisfy
\begin{eqnarray}\label{rho}
\begin{cases}
u(\rho_-)=u_p,\; u(\rho _+)=u_{p+1}\;, \;\rho_- <\rho _+;\\
\begin{cases}
\rho _t = a^2 \rho _{xx},\quad x\in \mathbb{R} ,\,t>-1,\\
\rho (x,t=-1)=
\begin{cases}
\rho _-, &x<0,\\
\rho _+, &x>0.
\end{cases}
\end{cases}
\end{cases}
\end{eqnarray}
Now, we define the $p$-th relaxation contact wave $u^p(x,t)$ as
$$u^p(x,t)\equiv u\big(\rho(x,t)\big)\in C_p(u_p).$$
Then it holds that,
$$u^p_t+f(u^p)_x-a^2 u^p_{xx}=- a^2 \nabla r_p\big(u(\rho)\big)\cdot r_p\big(u(\rho)\big)\rho_x^2=0,$$
provided that we impose the following structural condition:
\begin{equation}\label{sc}
\nabla r_p(u)\cdot r_p(u)\equiv 0,   \;\textrm{for} \;\; u\in
C_p(u_p).
\end{equation}
This structural condition was proposed first in \cite{lx} and used in \cite{hpw} to study the stability of contact waves. In addition to \eqref{sc}, we should note that in \cite{hpw} another structural condition:
\begin{equation*}\label{scleft}
\nabla l_p(u)\cdot r_p(u)\equiv 0,   \;\textrm{for} \;\; u\in
C_p(u_p),
\end{equation*}
was also imposed to show the stability of a single contact wave for the Jin-Xin model \eqref{lp1.1-1}.

 In order to study the asymptotic stability of superpositions of
shock waves and contact waves, we set the relaxation approximation to
the Riemann solution as the linear superposition of the
above two kinds of relaxation waves,
\begin{equation}\label{ubar}
\bar{u}(x,t)\equiv \sum\limits_{i=1}^n u^i(x,t)-(u_2+u_3+\cdots+u_n).
\end{equation}
Since for weak waves, $0<\delta_i\ll 1 \;(i=1,\cdots,n)$, the
vectors $u_2-u_1,\;u_3-u_2,\;\cdots,\;u_{n+1}-u_n$ form a basis of
$\mathbb{R}^n$, thus the initial mass can be decomposed into
\begin{equation}\label{uubarinitial}
\int_\mathbb{R}\big[u_0(x)-\bar{u}(x,0)\big]dx
=-\sum_{i=1}^n x_i(u_{i+1}-u_i),
\end{equation}
with the uniquely determined
constants $x_i\;(i=1,\cdots,n)$. The ansatz $u^a(x,t)$
is defined as
\begin{equation}\label{ua}
u^a(x,t)\equiv \sum\limits^n_{i=1} u^i(x-x_i,t)-(u_2+\cdots +u_n).
\end{equation}
It is easy to verify that $u^a$ satisfies
\begin{equation*}
\begin{split}
\int_\mathbb{R}\big[u_0(x)-u^a(x,0)\big]dx=&\int_\mathbb{R}\Big\{\big[u_0(x)-\bar{u}(x,0)\big]+
\big[\bar{u}(x,0)-u^a(x,0)\big]\Big\}dx \\=&-\sum_{i=1}^n
x_i(u_{i+1}-u_i)+\sum_{i=1}^n x_i(u_{i+1}-u_i)=0;
\end{split}\end{equation*}
and the following equation
\begin{equation}\label{uaequ}
u_t^a + f(u^a)_x - a^2 u_{xx}^a + u_{tt}^a
= (E_1 + E_2)_x,
\end{equation}
with error terms
\begin{equation}\label{E1E2}
 E_1 = f\big(u^a(x,t)\big) - \Bigg[\sum_{i = 1}^n f\big(u^i(x-x_i,t)\big) - \sum_{i=2}^n f(u_i)\Bigg] \ \  {\rm and} \ \
 E_2 = a^2 u^p_{xt}(x-x_i,t).
\end{equation}
Denote
\begin{equation*}
H(t) := \int_\mathbb{R} \big[u(x,t) - u^a(x,t)\big] dx \ \ {\rm with } \ \ H(0)=0.
\end{equation*}
It follows from \eqref{eqvjx} and \eqref{uaequ}that
\begin{equation*}
 H'(t) +  H''(t) = 0,
\end{equation*}
which implies
\begin{align*}
e^t   H'(t) = & \int_\mathbb{R} \big[u_t(x,0) - u^a_t(x,0)\big] dx
 = \int_\mathbb{R}\Big[-v_x(x,0) - \sum_{i=1}^n u^i_t(x-x_i,0)\Big] dx\\
 = & -(v_+ - v_-) + \big[f(u_+) - f(u_-)\big]
 =0.
\end{align*}
Hence, we have
\begin{equation*}
H(t) =  \int_\mathbb{R} \big[u(x,t) - u^a(x,t)\big] dx = 0
\quad
\mbox{for} ~\mbox{all} ~t\geqslant 0.
\end{equation*}
Define the perturbation $\phi(x, t)$ and the anti-derivative variable
$\Phi(x, t)$ as
\begin{equation*}
\phi(x,t) \equiv u(x,t) - u^a(x,t),
\quad
\Phi(x,t) \equiv \int_{-\infty}^x \phi(y,t)dy.
\end{equation*}
Note that $H(t)=0$ ensures that the anti-derivative variable $\Phi(x,t)$ is well defined
in some Sobolev spaces, like $L^2(\mathbb{R})$, $H^1(\mathbb{R})$, etc. It remains to construct the ansatz $v^a(x, t)$. Define
\begin{equation*}
v^a(x, t) \equiv f(u^a)  - a^2 u^a_{x } + \int_{-\infty}^x u^a_{tt} - E_{1} - E_{2} \ \ {\rm and} \ \  \psi(x, t)  \equiv v(x, t) - v^a(x, t).
\end{equation*}
Then we can get
\begin{equation*}
\phi_t + \psi_x = 0, \ \  \Phi_t = - \psi
\ \ {\rm and} \  \
 \Phi_x = \phi.
\end{equation*}

Now, the main result in this paper can be stated as follows.
\begin{theorem}\label{mainthm}
Let $\varepsilon =1$ and the sub-characteristic condition $\eqref{scc}$ hold. Assume that  \eqref{newlp1.1-4} is strictly hyperbolic and each characteristic
field is either genuinely nonlinear or linearly degenerate.
Suppose that the Riemann solution of \eqref{newlp1.1-4} and
\eqref{lp1.1-10} consists of one contact discontinuity with zero wave speed and $n-1$ shock waves,  whose strengths
satisfy $(\ref{strength})$ and $\delta_i>0$ ($i=1,\cdots,n$), and the structural condition \eqref{sc} holds.
Let the relaxation approximation $\bar{u}(x,t)$ and the ansatz $u^a(x,t)$ be defined by
$(\ref{ubar})$ and $(\ref{ua})$, respectively. Then there exists a small positive constant
$\delta_0$ such that if the wave strength and the initial value
 satisfy
\begin{equation}\label{ii1}
|u_+ - u_-| + \|\Phi(\cdot,0)\|_{H^3} + \|\psi(\cdot,0)\|_{H^2}
+ \delta_0^{3/2}\|u_0(\cdot)-\bar{u}(\cdot,0)\|_{L^1}
\leqslant \delta_0,
\end{equation}
then the problem $(\ref{lp1.1-1})$ with $(\ref{lp1.1-3})$ admits a unique global solution
$ (u, v)(x,t)$ satisfying
\begin{equation*}
u(x,t)\in C\big([0, +\infty); H^2\big) \cap L^2\big(0, +\infty; H^3\big),
\end{equation*}
\begin{equation*}
v(x,t)\in C\big([0, +\infty); H^1\big) \cap L^2\big(0, +\infty; H^2\big),
\end{equation*}
and
\begin{equation}
\|(u-u^a, v-v^a)(\cdot, t)\|_{L^\infty} \rightarrow 0, \quad \mbox{as} ~ t \rightarrow +\infty.
\end{equation}
\end{theorem}

We give several remarks on the theorem above. First, if $\delta_i=0$ for some integer $i\in[1,n]$, that
is, the rank of the waves of the Riemann solution to \eqref{newlp1.1-4} and
\eqref{lp1.1-10} is less than $n$, then our result also holds for zero mass
perturbation $(x_i=0, \;i=0,\cdots,n)$. Second, it follows from the proof that the result stated in Theorem \ref{mainthm} also holds for the Riemann solution consisting of only shock waves without contact discontinuities. Finally, the Riemann solution is assumed to contain one contact discontinuity and $n-1$ shock waves in this paper, due to the technique reason; the authors are working on a more general Riemann solution which consists of more than one contact discontinuity.

\textbf{Notations}.  Throughout this paper, generic positive constants
are denoted by $C$, $C_i$ ($i=1,2,3$) and $O(1)$ without confusion. These
constants depend only on the flux $f(u)$, the constant $a$ and the far field data. The functional space $H^l(\mathbb{R})$ denotes the $l$-th order sobolev
space with the norm
$$\|u\|_l=\sum_{j=0}^l\|\partial_x^j u\|, \qquad \textrm{when} \; \|\cdot\|=\|\cdot\|_{L^2(\mathbb{R})},$$
where $l$ is a non-negative integer.

\section{Preliminaries}
\quad ~
We list the following inequality based on the heat kernel for later use, whose proof can be found in  \cite{huang-li-matsumura,z}.

Let $\gamma$ be a positive constant and
\begin{equation}\label{hkg}
g(x,t) = (1 + t)^{- {1}/{2}} \int_{-\infty}^x \exp\big\{- \gamma y ^2/(1 + t)\big\} dy.
\end{equation}
It is easy to check that
\begin{equation}\label{gproperty}
g_t = {g_{xx}}/{(4\gamma)}, \quad \|g(\cdot, t)\| = \sqrt{\pi/\gamma}.
\end{equation}

\begin{lemma}\label{gh}
For $0 < T \leqslant +\infty$, suppose that $h(x,t)$ satisfies
\begin{equation*}
h \in L^{\infty}\big(0, T; L^2(\mathbb{R})\big), \quad
h_x \in L^2 \big(0, T; L^2(\mathbb{R})\big), \quad
h_t \in L^2 \big(0, T; H^{-1}(\mathbb{R})\big).
\end{equation*}
Then the following estimate holds for any $t\in (0, T]$,
\begin{equation*}
\int_0^t \int_\mathbb{R} h^2 g_x^2 dxdt
\leqslant 4\pi \|h(\cdot, 0)\|^2
+ \frac{4 \pi}{\gamma} \int_0^t \int_\mathbb{R} h_x^2 dxdt
+ 8\gamma\int_0^t <h_t, h g^2>_{H^{-1}\times H^1} dt.
\end{equation*}
\end{lemma}

The following two lemmas concern some properties of the relaxation contact waves (the proof is
clear, based on the classical heat kernel) and of the relaxation shock waves (the interested
reader may refer to \cite{lhl1,xzp} for the proofs).  It should be noted that the sub-characteristic condition \eqref{scc} plays an important role in deriving the properties of relaxation shock waves.

\begin{lemma}\label{lemmacontact}
The $p$-th contact wave  $u^p(x,t)=u(\rho(x,t)) \in C_p(u_p)$ satisfies the following properties:
1) the $p$-th characteristic speed is zero, i.e.,  $\lambda_p(u^p(x,t)) =0$; 2)  $u^p$ tends to the far field states exponentially fast, i.e.,
\begin{equation*}
\begin{split}
 & |u^p(x,t)-u_p|\leqslant O(1)\delta_p\exp\left\{-\frac{x^2}{8a^2(1+t)}\right\},\quad x<0,\\
  &      |u^p(x,t)-u_{p+1}|\leqslant O(1)\delta_p\exp\left\{-\frac{x^2}{8a^2(1+t)}\right\},\quad x>0;\\
 \end{split}
\end{equation*}
3) the spatial and temporal derivatives of $u^p$ decays for large time, i.e.,
 \begin{equation*}
\begin{split}
&|u^p_x|=\rho_x=O(1)\delta_p(1+t)^{- {1}/{2}}\exp\left\{-\frac{x^2}{4a^2(1+t)}\right\},\\
&|u^p_t|+|u^p_{xx}|\leqslant O(1)\delta_p(1+t)^{-1}\exp\left\{-\frac{x^2}{8a^2(1+t)}\right\},\\
& |u^p_{xt}|\leqslant O(1)\delta_p(1+t)^{-3/2}\exp\left\{-\frac{x^2}{16a^2(1+t)}\right\},\\
&|u^p_{tt}|+|u^p_{xxt}|\leqslant O(1)\delta_p(1+t)^{-2}\exp\left\{-\frac{x^2}{32a^2(1+t)}\right\}.\\
\end{split}
\end{equation*}
\end{lemma}

\begin{lemma}\label{lemshock}
For $i\neq p$, the relaxation shock wave
$u^i(x, t) = \varphi^i(\xi_i) = \varphi^i(x - s_i t)$
satisfies the following properties:
\begin{equation*}\begin{split}
(1)\quad&\frac{d}{d\xi_i}\lambda_i\big(\varphi^i(\xi_i)\big)<0, \quad{\rm for} \; \xi_i\in \mathbb{R};\\
(2)\quad&  \Big|\frac{d}{d
\xi_i}\lambda_i\big(\varphi^i(\xi_i)\big)\Big| =
O(1)\Big|\frac{d }{d\xi_i}\varphi^i(\xi_i)\Big|,\quad{\rm for} \; \xi_i\in \mathbb{R}; \\
(3)\quad &\int_{\mathbb{R}}\Big|\frac{d }{d
\xi_i}\lambda_i(\varphi^i(\xi_i))\Big|d\xi_i\leqslant O(1)\delta_i;\\
(4)\quad &\Big|\frac{d}{d \xi_i}\lambda_i(\varphi^i(\xi_i))\Big|\leqslant
O(1)\delta_i^2\exp\{-C\delta_i|\xi_i|\},\\
& \Big|\frac{d^2}{d \xi_i^2}\varphi^i(\xi_i)\Big|\leqslant O(1)\delta_i\Big|\frac{d}{d
\xi_i}\varphi^i(\xi_i)\Big|,\quad{\rm for} \; \xi_i\in \mathbb{R};\\
(5)\quad&|\varphi^i(\xi_i)-u_i|\leqslant
O(1) \delta_i \exp\{-C\delta_i|\xi_i|\},\quad \xi_i<0,\\
&|\varphi^i(\xi_i)-u_{i+1}|\leqslant O(1) \delta_i
\exp\{-C\delta_i|\xi_i|\},\quad \xi_i>0.
\end{split}\end{equation*}
\end{lemma}

To deal with the wave interactions from the different characteristic
fields, divide $\mathbb{R}\times(0,t)$ into $n$ parts as
$\mathbb{R}\times(0,t)=\Omega_1\cup\Omega_2\cup\cdots\cup\Omega_n$,
where
\begin{equation*}
\begin{split}
&\Omega_1=\left\{(x,t)\Big|\;x\leqslant \frac{s_1+s_2}{2}t\;\right\},\\
&\Omega_i=\left\{(x,t)\Big|\;\frac{s_{i-1}+s_i}{2}t< x\leqslant \frac{s_{i}+s_{i+1}}{2}t\;\right\},\quad i=2,\cdots,n-1,\\
&\Omega_n=\left\{(x,t)\Big|\;x> \frac{s_{n-1}+s_n}{2}t\;\right\}.\\
\end{split}
\end{equation*}
We also divide $\mathbb{R}\times(0,t)$ as
$\mathbb{R}\times(0,t)=\Omega_i^-\cup\Omega_i\cup\Omega_i^+$ for
$i=1,\cdots,n$, where
\begin{equation*}
\begin{split}
&\Omega_i^+=\mathop{\cup}\limits_{j>i}
\Omega_j=\left\{(x,t)\Big|\;x\geqslant\frac{s_{i}+s_{i+1}}{2}t\;\right\},\quad
i=1,\cdots,n-1,\\
&\Omega_i^-=\mathop{\cup}\limits_{j<i} \Omega_j=\left\{(x,t)\Big|\;x\leqslant\frac{s_{i}+s_{i-1}}{2}t\;\right\},\quad
i=2,\cdots,n, \\
&\Omega_1^-=\Omega_n^+=\emptyset.
\end{split}
\end{equation*}
Then we can get the following lemma, whose proof can be found in \cite{z}.
\begin{lemma}\label{leme}
Set
\begin{equation}\label{t0}
t_0=4\max\limits_{i=1,\cdots,n}\{|x_i|\}/\min\limits_{i=2,\cdots,n}\{s_i-s_{i-1},s\},
\end{equation}
where $s$ is the minimum positive wave speed, otherwise, $s=+\infty$. When $t>t_0$,
there exists a constant $c_0(>0)$, depending only on the flux and the far field data, such that\\
(1) for $i=p$,
\begin{equation*}\label{p10}\left\{
\begin{split}
&|u^i(x-x_i,t)-u_i| \leqslant O(1)\delta_i\exp\{-c_0(t+|x|)\}, &\textrm{in}
\quad\Omega_i^-,\\
&|u^i(x-x_i,t)-u_{i+1}|\leqslant O(1)\delta_i\exp\{-c_0(t+|x|)\},
&\textrm{in}
\quad\Omega_i^+,\\
&|u^i_x(x-x_i,t)|\leqslant O(1)\delta_i(1+t)^{-1/2}\exp\{-c_0(t+|x|)\},
&\textrm{in}
\quad\Omega_i^c;\\
\end{split}\right.
\end{equation*}
(2) for $i\neq p$,
\begin{equation*}\label{p11}\left\{
\begin{split}
&|u^i(x-x_i,t)-u_i|\leqslant O(1)\delta_i\exp\{-c_0\delta_i(t+|x|)\},
&\textrm{in}
\quad\Omega_i^-,\\
&|u^i(x-x_i,t)-u_{i+1}|\leqslant O(1)\delta_i\exp\{-c_0\delta_i(t+|x|)\},
&\textrm{in}
\quad\Omega_i^+,\\
&|u^i_x(x-x_i,t)|\leqslant O(1)\delta_i^2\exp\{-c_0\delta_i(t+|x|)\},
&\textrm{in}
\quad\Omega_i^c;\\
\end{split}\right.
\end{equation*}
(3) assume further that \eqref{strength} holds, then the ansatz $u^a(x,t)$
satisfies
\begin{equation*}\label{p12}\begin{split}
&u^a_t + f(u^a)_x-u^a_{xx}=(E_1+E_2)_x \  \
 \textrm{with} \ \ |E_1|\leqslant O(1)\delta^2\exp\{-c_0\delta(t+|x|)\}, \\
&
|u^a(x,t)-u^i(x-x_i,t)||u^i_x(x-x_i,t)|\leqslant O(1)\delta^2\exp\{-c_0\delta(t+|x|)\}.
\end{split}\end{equation*}
\end{lemma}
It should be noted that $t_0$ depends on the shifts of the wave
locations. The assumptions of strict hyperbolicity and the initial data \eqref{ii1}
imply a bound on $t_0$. Denote
\begin{eqnarray}\label{ee}
e=
\begin{cases}
O(1) \delta^2  \exp\{-c_0 \delta (t+|x|)\}, \quad &t>t_0,\\
O(1) \delta^2 \bigg[
\exp\Big\{-\frac{(x-x_p)^2}{8(1+t)}\Big\}
+\sum\limits_{i=1 \atop i \neq p}^n\exp\big\{-C\delta |x-s_it-x_i|\big\}
\bigg], \quad & t\leqslant t_0.
\end{cases}
\end{eqnarray}
One can easily get that
\begin{equation}\label{E1}
|E_1| \leqslant  e \quad \mbox{and} \quad |u^a-u^i||u^i_x|\leqslant e \ \ {\rm for} \ \  i = 1, \cdots, n.
\end{equation}

\section{Stability analyses}
\quad ~
In this section, we will prove Theorem $\ref{mainthm}$. The lower-order estimate is
derived in section $\ref{secloe}$,
which consists of a weighted energy estimate and a estimate based on Lemma $\ref{gh}$.
The higher-order estimates are presented in section $\ref{sechoe}$ and the asymptotic behavior is shown in the last subsection.

It follows from \eqref{eqvjx} and $(\ref{uaequ})$ that $\phi(x, t)=(u-u^a)(x,t)$
satisfies
\begin{equation}\label{lo1}
\phi_t + \big[f(u) - f(u^a)\big]_x - a^2 \phi_{xx} + \phi_{tt} = -(E_1 + E_2)_x.
\end{equation}
Integrating $(\ref{lo1})$ with respect to the spatial variable from $-\infty$ to $x$, one gets
\begin{equation}\label{lo3}
\Phi_t +   f'(u^a) \Phi_x - a^2 \Phi_{xx} + \Phi_{tt}
= - E_1 - E_2 + Q,
\end{equation}
where
\begin{equation}\label{Q}
|Q| = |f(u^a + \Phi_x) - f(u^a) -  f'(u^a) \Phi_x|
\leqslant O(1) |\Phi_x|^2.
\end{equation}
We will work on the Cauchy problem of $(\ref{lo3})$ with the initial data:
\begin{equation}\label{initial}
\Phi(x, 0) = \int_{-\infty}^x \big(u - u^a\big)(y, 0) dy \ \ {\rm and} \ \
\Phi_t(x, 0) = (v^a-v)(x, 0).
\end{equation}

Notice that the standard theory gives the local existence and uniqueness of classical solutions to $(\ref{lo3})$ and $(\ref{initial})$ for some short time $T^*$.
In order to obtain the global existence and further to study the large time
asymptotic behavior, we need to close the following a priori assumption
\begin{equation}\label{pa}
N(T) =
\sup_{0\leqslant t\leqslant T}\big(
\|\Phi\|_{H^3} + \|\Phi_t\|_{H^1}
\big) \leqslant \varepsilon_0,
\end{equation}
where the small positive constant $\varepsilon_0$ only depends on
the initial data and the wave strength. Clearly, $(\ref{pa})$ is true for a short time if we
choose $\delta_0$ small, due to local theory. We will prove that $T=+\infty$ with the help of
uniform estimates and the standard continuation argument.

\subsection{Lower-order estimates}\label{secloe}
\quad ~
To diagonalize  system $(\ref{lo3})$, we introduce a new variable
\begin{equation}\label{WPhi}
W(x,t) \equiv L\big(u^a(x,t)\big) \Phi(x,t).
\end{equation}
Then $
\Phi = R(u^a) W
$ and system \eqref{lo3} reads
\begin{equation}\label{lo6}
\begin{split}
W_t + \Lambda(u^a) W_x - a^2 W_{xx} + W_{tt} = - A W + B,
\end{split}
\end{equation}
where
\begin{align}
&A = (a_{ij})_{n \times n}
:= L(u^a) \big[R(u^a)_t - a^2 R(u^a)_{xx} + f'(u^a) R(u^a)_x + R(u^a)_{tt}\big]\label{B},
\\
&B = (b_i)_{n \times 1}:=L(u^a)\big[2 a^2 R(u^a)_x W_x - E_1 - E_2 + Q - 2 R(u^a)_t W_t\big]\label{D}.
\end{align}
As in the stability analysis for  a single relaxation waves (\emph{cf.}\cite{hpw,lhl1}),
we first  construct the weighted functions.  Set
\begin{equation}\label{eta}
\eta(x, t) \equiv  {\rho(x, t)}/{\rho_+}  ,
\end{equation}
where $\rho(x,t)$, $\rho_+$ are given by $(\ref{rho})$; and define
\begin{equation}\label{defialphaic}
\alpha_i^c = \eta ^m ~~\mbox{for} ~i < p,
\quad\quad
\alpha_p^c = 1,
\quad\quad
\alpha_i^c = \eta^{-m} ~~\mbox{for} ~i > p,
\end{equation}
with $m = \delta^{-{1}/{2}}$ here and thereafter. Then  $\alpha_i^c$ is used to take account of the relaxation contact wave. To deal with the relaxation shock waves, we choose
$$\alpha_i^s\equiv \sum_{j=1 \atop j \neq p}^n \beta_i^j(\varphi^j)
= \sum_{j=1 \atop j \neq p}^n \beta_i^j \big(\varphi^j(\xi_j)\big)
= \sum_{j=1 \atop j \neq p}^n \beta_i^j \big(\varphi^j(x - s_j t)\big),$$
where
\begin{equation*}\label{w2}
\begin{split}
\beta_i^j (\varphi^j) =
\begin{cases}
\frac{\lambda_i(\varphi^j(0)) - s_j}{\lambda_i(\varphi^j) - s_j}
\exp \bigg\{\displaystyle{- m \int_0^{\xi_j}
\frac{| \partial_{\xi_j} \lambda_j (\varphi^j) |}
{\lambda_i(\varphi^j) - s_j}
d\xi_j }\bigg\}, \quad & j \neq p, i,\\
1, \quad & j \neq p, j = i.
\end{cases}
\end{split}
\end{equation*}
A simple calculation shows that $\beta_i^j(\varphi^j)$ satisfies
\begin{equation}\label{lo9}
\partial_{\xi_j} \Big\{ \big[ \lambda_i(\varphi^j) - s_j \big] \beta_i^j (\varphi^j) \Big\}
= - m \beta_i^j (\varphi^j) \big| \partial_{\xi_j}  \lambda_j(\varphi^j) \big|,
\quad j \neq p,i.
\end{equation}
We choose the weight matrix $\alpha \equiv \mbox{diag} \{\alpha_1,\cdots, \alpha_n\}$ with
\begin{equation}\label{alphai}
\alpha_i \equiv \alpha_i^c + \alpha_i^s, \quad i = 1, \cdots, n.
\end{equation}
The assumption of the strict hyperbolicity, Lemma \ref{lemmacontact}
and Lemma \ref{lemshock} yields the bounds for these weighted functions
\begin{equation}\label{alphaiproperty}
1-C\delta^{1/2} \leqslant \alpha_i \leqslant 1+C\delta^{1/2},
 \quad i = 1,\cdots, n,
\end{equation}
for some positive constant $C$. In fact, \eqref{eta} gives
\begin{equation*}
0<\eta\leqslant 1, \quad\quad
|\eta - 1|\leqslant O(1) \delta_p,
\end{equation*}
then, if $\delta_p$ is small enough, we can get
\begin{equation}\label{alphaicproperty}
1-O(1)\delta_p^{1/2} \leqslant \eta^m \leqslant 1 \leqslant \eta^{-m}
\leqslant 1+O(1)\delta_p^{1/2}.
\end{equation}
And, in view of the strict hyperbolicity and Lemma \ref{lemshock}, when $j\neq p, i$, we have
\begin{equation*}
\bigg|- m \int_0^{\xi_j}
\frac{| \partial_{\xi_j} \lambda_j (\varphi^j) |}{\lambda_i(\varphi^j) - s_j}d\xi_j\bigg|
\leqslant O(1) m \int_{\mathbb{R}}\Big|\frac{d }{d
\xi_j}\lambda_j(\varphi^j)\Big|d\xi_j\leqslant O(1)\delta^{1/2},
\end{equation*}
so, by the Taylor expansion of the exponential function, we obtain
\begin{equation}\label{alphaisproperty}
1-O(1)\delta^{1/2} \leqslant \beta_i^j(\varphi^j)\leqslant 1+O(1)\delta^{1/2}.
\end{equation}
Thus, \eqref{alphaicproperty} and \eqref{alphaisproperty} give \eqref{alphaiproperty}.

\subsubsection{Basic energy estimate}
\quad ~
Now, we begin to derive the weighted energy estimates for system \eqref{lo6}.  Multiplying system $(\ref{lo6})$ by $W^T\alpha=(w_1\alpha_1, \cdots, w_n \alpha_n)$ and integrating the product with respect to $x$ over $\mathbb{R}$, we get
\begin{equation}\label{djgj1}
\begin{split}
&\frac{1}{2}\frac{d}{dt} \int_\mathbb{R}
\sum_{i=1}^n\big(
\alpha_i w_i^2 +   2\alpha_i w_i w_{it}
\big) dx
+ a^2 \int_\mathbb{R} \sum_{i=1}^n \alpha_i w_{ix}^2dx
- \int_\mathbb{R} \sum_{i=1}^n \alpha_i w_{it}^2 dx
+  \rm{J}\\
=&\rm{I}_1 + \rm{I}_2,
\end{split}
\end{equation}
where
\begin{equation}\label{j}
{\rm{J}} = - \frac{1}{2}  \int_\mathbb{R}  \sum_{i=1}^n
\Big\{\alpha_{it} + \big[\alpha_i \lambda_i(u^a)\big]_x \Big\} w_i^2 dx
+\int_\mathbb{R} \sum_{i=1}^n \big( a^2 \alpha_{ix}w_{ix}- \alpha_{it} w_{it} \big)  w_i dx,
\end{equation}
\begin{equation}\label{i1i2}
{\rm{I}}_1 = -  \int_\mathbb{R}  W^T \alpha A W dx,
\quad\quad
{\rm{I}}_2 =   \int_\mathbb{R}  W^T \alpha B dx,
\end{equation}
where $A$ and $B$ are defined in \eqref{B} and \eqref{D}.
Next, we will deal with these three terms $\rm{J}$, $\rm{I}_1$ and $\rm{I}_2$ in equation \eqref{djgj1}.

\begin{lemma}\label{lemJ}
Under the same assumptions as in
Theorem $\ref{mainthm}$, it holds that
\begin{equation}\label{J}
\begin{split}
\rm{J} &\geqslant
O(1) m \int_\mathbb{R} \sum_{i,j = 1 \atop j \neq i}^n |u^j_x| w_i^2 dx
+ \frac{1}{2} \int_\mathbb{R} \sum_{i=1 \atop i \neq p}^n
\alpha_i |\lambda_i(u^i)_x| w_i^2 dx
- O(1) \int_\mathbb{R} e |W|^2 dx\\
&\quad- O(1) \delta^{{1}/{4}} \big(\|W_x\|^2 + \|W_t\|^2\big)
.
\end{split}
\end{equation}
Here the error term $e$, $\alpha_i$ and $J$ are given by $(\ref{ee})$,  $(\ref{alphai})$
and \eqref{j}
respectively.
\end{lemma}
\begin{proof} By direct calculation, one can get
\begin{equation}\label{jgj}
 \alpha_{it} + \big[\alpha_i  \lambda_i(u^a) \big]_x
=  \big[\alpha_{it}^s + \alpha_{ix}^s \lambda_i(u^a)\big]
+ \alpha_{ix}^c  \lambda_i(u^a)
+ \alpha_i \lambda_i(u^a)_x
+ \alpha_{it}^c.
\end{equation}
Next, we will derive estimates about these four terms one by one as follows.
First, according to the definition of $\alpha_i^s$, $(\ref{lo9})$ and the smallness
of $\delta$, we can get
\begin{equation*}
\begin{split}
\alpha_{it}^s + \alpha_{ix}^s \lambda_i(u^a)
=& \sum_{j=1 \atop j \neq p}^n
\big[\lambda_i(u^a) - s_j\big] \partial_{\xi_j}\beta_i^j (\varphi_j)\\
=& \sum_{j=1 \atop j\neq p,i}^n
\frac{\lambda_i(u^a) - s_j}{\lambda_i(\varphi_j) - s_j}  \beta_i^j (\varphi_j)
\big[- m |\partial_{\xi_j}  \lambda_j(\varphi_j)|
-  \partial_{\xi_j}  \lambda_i(\varphi_j)\big]\\
\leqslant& \sum_{j=1 \atop j \neq p,i }^n O(1) \big(- m  |u^j_x|
+ |u^j_x|\big)
\leqslant - O(1) m  \sum_{j=1 \atop j \neq p,i }^n |u^j_x|.
\end{split}
\end{equation*}
Then, the above estimate gives
\begin{equation}\label{lp4.2}
\begin{split}
\frac{1}{2} \sum_{i = 1}^n \big[\alpha_{it}^s + \alpha_{ix}^s \lambda_i(u^a)\big] w_i^2
\leqslant - O(1) m \sum_{i,j = 1 \atop j \neq p,i }^n |u^j_x| w_i^2 .
\end{split}
\end{equation}
From the definition \eqref{defialphaic} of $\alpha_i^c$, we know,
\begin{equation*}
\alpha_{ix}^c = m\eta^{m-1}\eta_x ~~\mbox{for} ~i<p,
\quad\quad
\alpha_{px}^c=0,
\quad\quad
\alpha_{ix}^c = -m\eta^{-m-1}\eta_x ~~\mbox{for}~ i>p.
\end{equation*}
Then, due to the strict hyperbolicity of $(\ref{lp1.1-4})$, we have
\begin{equation}\label{alphaixc}
\begin{split}
\frac{1}{2}  \sum_{i=1}^n \alpha_{ix}^c \lambda_i(u^a) w_i^2
&= \frac{1}{2} m   \bigg[
\sum_{i=1 \atop i<p}^n \eta^{m-1} \lambda_i(u^a) w_i^2
-
\sum_{i=1 \atop i>p}^n \eta^{-m-1} \lambda_i(u^a) w_i^2
\bigg] \eta_x  \\
&\leqslant - O(1)m  \sum_{i=1 \atop i\neq p}^n \rho_x w_i^2
\leqslant - O(1)m \sum_{i=1 \atop i\neq p}^n |u^p_x| w_i^2 .
\end{split}
\end{equation}
Notice that
\begin{equation*}
\begin{split}
\lambda_i(u^a)_x
= \sum_{j=1}^n \nabla \lambda_i(u^a) u_x^j
&= \nabla\lambda_i(u^i)u^i_x
+ \big[\nabla \lambda_i(u^a) - \nabla \lambda_i(u^i)\big] u^i_x
+\sum_{j=1 \atop j \neq i}^n\nabla \lambda_i(u^a )u^j_x\\
&\leqslant -|\lambda_i(u^i)_x|\chi_{\{i \neq p\}} + O(1)
\bigg(e + \sum_{j=1 \atop j \neq i}^n |u^j_x| \bigg),
\end{split}
\end{equation*}
where $\chi$ is the characteristic function with
\begin{equation*}
\chi_{\{i \neq j\}}= 1 ~~\mbox{for} ~ i \neq j,
\quad\quad
\chi_{\{i \neq j\}}= 0 ~~\mbox{for} ~ i = j,
\end{equation*}
and $e$ is defined by $(\ref{ee})$.
Then, we can obtain the estimate
\begin{eqnarray}\label{lo14}
\begin{split}
\frac{1}{2} \sum_{i = 1}^n
\alpha_i \lambda_i(u^a)_x w_i^2
\leqslant
- \frac{1}{2}  \sum_{i=1 \atop i \neq p}^n
\alpha_i |\lambda_i(u^i)_x| w_i^2
+  O(1)\bigg( \sum_{i,j = 1 \atop j \neq i}^n |u^j_x| w_i^2
+  e |W|^2\bigg) .
\end{split}
\end{eqnarray}
From \eqref{rho} we get $\eta_t = a^2 \eta_{xx}$, this fact together
with \eqref{defialphaic} gives
\begin{equation*}
\alpha_{it}^c = a^2 m \eta^{m-1}\eta_{xx} ~~\mbox{for} ~i<p,
\quad\quad
\alpha_{pt}^c=0,
\quad\quad
\alpha_{it}^c = -a^2 m \eta^{-m-1}\eta_{xx} ~~\mbox{for} ~i>p,
\end{equation*}
because we also have the following two inequalities:
\begin{equation*}
\eta^{m-1} \eta_{xx}
=  (\eta^{m-1} \eta_x)_x - (m-1) \eta^{m-2} \eta_x^2
\leqslant (\eta^{m-1} \eta_x)_x,
\end{equation*}
\begin{equation*}
 \eta^{-m-1} \eta_{xx}
=  (\eta^{-m-1} \eta_x)_x +  (m+1) \eta^{-m-2} \eta_x^2
\geqslant (\eta^{-m-1} \eta_x)_x.
\end{equation*}
Then using the integration by parts and the Cauchy-Schwartz inequality, we have
\begin{equation}\label{alphaixt}
\begin{split}
\frac{1}{2}  \int_\mathbb{R} \sum_{i=1}^n \alpha_{it}^c w_i^2dx
&\leqslant \frac{1}{2} a^2 m   \int_\mathbb{R} \bigg[
\sum_{i=1 \atop i<p}^n ( \eta^{m-1} \eta_x)_x w_i^2
-\sum_{i=1 \atop i>p}^n(\eta^{-m-1} \eta_x)_x w_i^2
\bigg]dx\\
&= - a^2 m  \int_\mathbb{R}  \bigg(
\sum_{i=1 \atop i<p}^n \eta^{m-1} \eta_x  w_i w_{ix}
- \sum_{i=1 \atop i>p}^n \eta^{-m-1} \eta_x w_i w_{ix}
\bigg)dx\\
& \leqslant O(1)  m   \int_\mathbb{R} \sum_{i=1 \atop i \neq p}^n
\rho_x |w_i| |w_{ix}|dx\\
&= O(1)  m  \int_\mathbb{R} \sum_{i=1 \atop i \neq p}^n
\Big(\rho_x^{{5}/{8}} |w_i|\Big) \Big(\rho_x^{{3}/{8}}|w_{ix}|\Big)dx\\
&\leqslant O(1)  m \delta^{{1}/{4}}  \int_\mathbb{R} \sum_{i=1 \atop i \neq p}^n
|u^p_x| w_i^2dx
+ O(1) \delta^{{1}/{4}} \|W_x\|^2 .
\end{split}
\end{equation}
Adding \eqref{lp4.2}-\eqref{lo14} together and integrating
on $\mathbb{R}$ with respect to $x$, the resulting inequality
together with \eqref{alphaixt} yields
\begin{equation}\label{lo16}
\begin{split}
&- \frac{1}{2}  \int_\mathbb{R} \sum_{i = 1}^n
\Big\{ \alpha_{it} +  \big[ \alpha_i\lambda_i(u^a)\big]_x \Big\} w_i^2dx
\geqslant   O(1) m  \int_\mathbb{R} \sum_{i,j = 1 \atop j \neq i}^n |u^j_x| w_i^2dx\\
&\quad+ \frac{1}{2}  \int_\mathbb{R} \sum_{i=1 \atop i \neq p}^n
\alpha_i |\lambda_i(u^i)_x| w_i^2 dx
- O(1) \int_\mathbb{R}  e |W|^2 dx - O(1)\delta^{{1}/{4}} \|W_x\|^2.
\end{split}
\end{equation}
Similarly, using the Cauchy-Schwartz inequality, we can get
\begin{align*}
& \int_\mathbb{R} \sum_{i=1}^n ( a^2 \alpha_{ix}w_{ix}- \alpha_{it} w_{it} )  w_i dx\\
=& a^2 \int_\mathbb{R} \sum_{i=1}^n (\alpha_{ix}^c + \alpha_{ix}^s) w_i w_{ix} dx
- \int_\mathbb{R} \sum_{i=1}^n (\alpha_{it}^c + \alpha_{it}^s) w_i w_{it} dx\\
\geqslant& - O(1)  m \delta^{{1}/{4}}
 \int_\mathbb{R} \sum_{i,j=1 \atop  j \neq i}^n |u^j_x| w_i^2dx
- O(1) \delta^{{1}/{4}} \big(\|W_x\|^2 + \|W_t\|^2\big) .
\end{align*}
The desired estimate $(\ref{J})$ is obtained by adding the last inequality and $(\ref{lo16})$
together.
\end{proof}

\begin{lemma}\label{lemI1}
Under the same assumptions as in
Theorem $\ref{mainthm}$, it holds that
\begin{equation}\label{I1}
\begin{split}
{\rm{I}}_1 =& -  \int_\mathbb{R}  W^T \alpha A W dx
\leqslant  O(1) \int_\mathbb{R}
\sum_{i,j = 1\atop j \neq i}^n |u^j_x| w_i^2 dx
+ O(1) \delta^{{1}/{2}} \int_\mathbb{R} \sum_{i=1 \atop i \neq p }^n |u^i_x| w_i^2 dx\\
&+ O(1) \delta^{-{1}/{2}} \int_\mathbb{R} |u^p_x|^2 w_p^2 dx
+ O(1) \bigg[ \int_\mathbb{R} e |W|^2 dx
+ \varepsilon_0 \delta (1+t)^{-{3}/{2}}\bigg].
\end{split}
\end{equation}
Here the error term $e$, the matrix $A$ and $\rm{I}_1$ are given by $(\ref{ee})$, \eqref{B}
and \eqref{i1i2} respectively.
\end{lemma}
\begin{proof}
From the definition of $A$ in \eqref{B} and of $\rm{I}_1$ in \eqref{i1i2}, we have
\begin{equation}\label{i1sum}
\begin{split}
{\rm{I}}_1 = -  \int_\mathbb{R}  \sum_{i,j=1}^n w_i\alpha_i a_{ij}w_j dx
\leqslant O(1) \int_\mathbb{R}  \sum_{i,j=1}^n |a_{ij}| | w_iw_j | dx.
\end{split}
\end{equation}
First, we will give a detailed computation on $a_{ij}$. In view of \eqref{uaequ}, we can get
\begin{align*}
a_{ij}
&= l_i(u^a) \big[r_j(u^a)_t - a^2 r_j(u^a)_{xx} + \lambda_i(u^a) r_j(u^a)_x  + r_j(u^a)_{tt}\big]\nonumber\\
&= l_i(u^a) \nabla r_j(u^a) \big[u_t^a - a^2 u^a_{xx} + \lambda_i(u^a) u^a_x + u^a_{tt}\big]
- l_i(u^a) \nabla^2 r_j(u^a) \big[(u^a_x, u^a_x) - (u^a_t, u^a_t)\big]\nonumber\\
&\leqslant l_i(u^a) \nabla r_j(u^a) \big\{E_{1x} + E_{2x}
+ \big[\lambda_i(u^a)\mathbb{I} - f'(u^a)\big] u^a_x\big\}
+ O(1) \big( |u^a_x|^2 + |u^a_t|^2\big).
\end{align*}
The spatial derivative of $E_1$ gives
\begin{equation}\label{E1x}
E_{1x} = f(u^a)_x - \sum_{i=1}^n f(u^i)_x
= \sum_{i=1}^n \big[f'(u^a) - f'(u^i)\big] u^i_x
\leqslant O(1)e,
\end{equation}
where $E_1$ is defined in \eqref{E1E2} and we have used \eqref{E1} which will be used repeatedly in the next steps and will not
be mentioned. According to the structural condition \eqref{sc},
a detailed computation for $E_{2x}$ gives
\begin{equation}\label{E2x}
E_{2x} = a^2 u^p_{xxt}
=a^2 \big(u^p_x\big)_{xt}
=a^2 \big[r_p(u^p) \rho_x\big]_{xt}
= a^2 \big[r_p(u^p) \rho_{xx}\big]_t
=a^2 r_p(u^p) \rho_{xxt},
\end{equation}
then, for the term including $E_{2x}$, we have
\begin{equation*}\label{e2x1}
\begin{split}
l_i(u^a) \nabla r_j(u^a) E_{2x}
= a^2  l_i(u^a) \nabla r_j(u^a) r_p(u^p) \rho_{xxt}
\leqslant O(1) |\rho_{xxt}|,
\end{split}
\end{equation*}
where $E_2$ is also defined in \eqref{E1E2}. So far, the estimate for $a_{ij}$ reads
\begin{equation}\label{bij1}
a_{ij}
\leqslant  l_i(u^a) \nabla r_j(u^a) \big[\lambda_i(u^a)\mathbb{I} - f'(u^a)\big] u^a_x
+ O(1) \big( e + |u^a_x|^2 + |u^a_t|^2 + |\rho_{xxt}| \big).
\end{equation}
Next, we will derive the estimate for the first term in \eqref{bij1}.
A direct calculation gives
\begin{align*}
&l_i(u^a) \nabla r_j(u^a) \big[\lambda_i(u^a)\mathbb{I} - f'(u^a)\big] u^a_x\nonumber\\
=& \sum_{k=1}^n l_i(u^a) \nabla r_j(u^a)\Big\{ \big[\lambda_i(u^k)\mathbb{I} - f'(u^k)\big]
+ \big[\lambda_i(u^a) - \lambda_i(u^k)\big]
- \big[f'(u^a) - f'(u^k)\big] \Big\} u^k_x\nonumber\\
\leqslant& \sum_{k=1}^n l_i(u^a) \nabla r_j(u^a) \big[\lambda_i(u^k)\mathbb{I} - f'(u^k)\big] u^k_x
+ O(1) e.
\end{align*}
The first term in the above inequlity will be estimated from two sides. When $k\neq p$, considering the equation \eqref{shequ} satisfied by
the smooth traveling wave, we have
\begin{align*}
l_i(u^a) \nabla r_j(u^a) \big[\lambda_i(u^k)\mathbb{I} - f'(u^k)\big] u^k_x
=& l_i(u^a) \nabla r_j(u^a)
\Big\{ \big[\lambda_i(u^k) - s_k\big] u^k_x - (a^2 - s_k^2) u^k_{xx} \Big\}\\
\leqslant &
O(1) \big( |\lambda_i(u^k) - s_k||u^k_x|
+ |u^k_{xx}|\big)
\leqslant O(1)\delta |u^k_x|.
\end{align*}
When $k=p$, considering the relaxation contact wave speed is zero and
$u^p_x=r_p(u^p)\rho_x$, we have
\begin{align*}
l_i(u^a) \nabla r_j(u^a) \big[\lambda_i(u^p)\mathbb{I} - f'(u^p)\big] u^p_x
=&\lambda_i(u^p) l_i(u^a) \nabla r_j(u^a) u^p_x\\
=& \lambda_i(u^p) l_i(u^a)
\Big\{\nabla r_j(u^p) + \big[\nabla r_j(u^a) - \nabla r_j(u^p)\big]\Big\} r_p(u^p)\rho_x\\
\leqslant&
O(1) \big(
|u^p_x| \chi_{\{i,j \neq p\}}
+e
\big),
\end{align*}
where \eqref{sc} has been used again. Then, by the above two inequalities, we have
\begin{equation}\label{f1}
\begin{split}
l_i(u^a) \nabla r_j(u^a) \big[\lambda_i(u^a)\mathbb{I} - f'(u^a)\big] u^a_x
\leqslant O(1)\bigg(
\delta \sum_{k=1 \atop k \neq p}^n |u^k_x|
+|u^p_x| \chi_{\{i,j \neq p\}}
+ e \bigg).
\end{split}
\end{equation}
By the definition \eqref{ua} of $u^a$, one can get
\begin{equation}\label{uax2}
|u^a_x|^2 \leqslant O(1) \sum_{k=1}^n |u^k_x|^2
= O(1) \bigg( \sum\limits_{k=1 \atop k \neq p}^n |u^k_x|^2 + |u^p_x|^2 \bigg)
\leqslant O(1) \bigg(\delta \sum\limits_{k=1 \atop k \neq p}^n |u^k_x| + |u^p_x|^2  \bigg).
\end{equation}
Noticing  $|u^a_t|\leqslant O(1) |u^a_x|$, substitute \eqref{f1} and \eqref{uax2}
into \eqref{bij1}, we obtain
\begin{equation}\label{bij3}
\begin{split}
|a_{ij}|
\leqslant O(1)\bigg(
\delta \sum_{k=1 \atop k \neq p}^n |u^k_x|
+|u^p_x| \chi_{\{i,j \neq p\}}
 + |u^p_x|^2 + |\rho_{xxt}|+ e
\bigg).
\end{split}
\end{equation}
On the basis of \eqref{i1sum} and \eqref{bij3}, we have
\begin{equation}\label{i1bij}
\begin{split}
{\rm{I}}_1
\leqslant O(1) \int_\mathbb{R}  \sum_{i,j=1}^n
\bigg(\delta \sum_{k=1 \atop k \neq p}^n |u^k_x|
+|u^p_x| \chi_{\{i,j \neq p\}}
+ |u^p_x|^2 + |\rho_{xxt}|+ e\bigg) | w_iw_j | dx.
\end{split}
\end{equation}
Next, we will calculate the terms of the integrand in \eqref{i1bij}. Applying the
Cauchy-Schwartz inequality, we can get
\begin{equation*}
\begin{split}
\delta \sum_{i,j,k=1 \atop k \neq p}^n |u^k_x||w_iw_j|
\leqslant O(1) \delta \sum_{i,j,k=1 \atop k \neq p}^n |u^k_x| (w_i^2 + w_j^2)
\leqslant O(1) \bigg(\sum_{i,j=1 \atop j \neq p,i}^n |u^j_x| w_i^2
+ \delta^{1/2} \sum_{i=1 \atop i\neq p}^n |u^i_x| w_i^2\bigg).
\end{split}
\end{equation*}
Similarly, we also have
\begin{equation*}
\begin{split}
&\sum_{i,j=1}^n \big(
|u^p_x| \chi_{\{i,j \neq p\}}
+ |u^p_x|^2 \big)|w_iw_j|
\leqslant O(1) \bigg(\sum_{i=1 \atop i\neq p}^n |u^p_x| w_i^2
+ \delta ^{-1/2} |u^p_x|^2 w_p^2\bigg)\end{split}.
\end{equation*}
The a priori assumption \eqref{pa} yields the fact
$\|W\|_{L^\infty} \leqslant O(1) \|W\|_1 \leqslant O(1) \varepsilon_0$.
Using this fact, we can get
\begin{equation}\label{rhoxxt}
\int_\mathbb{R}\sum_{i,j=1}^n |\rho_{xxt}||w_iw_j|dx
\leqslant O(1)\varepsilon_0^2 \int_\mathbb{R}|\rho_{xxt}|dx
\leqslant  O(1) \varepsilon_0 \delta (1+t)^{-{3}/{2}}.
\end{equation}
Thus we finish the proof by the last three inequalities and \eqref{i1bij}.
\end{proof}

\begin{lemma}\label{lemI2}
Under the same assumptions as in
Theorem $\ref{mainthm}$, it holds that
\begin{equation}\label{I2}
\begin{split}
\rm{I}_2 =&  \int_\mathbb{R}  W^T \alpha B dx
\leqslant
O(1) \int_\mathbb{R} \sum_{i,j = 1\atop j \neq i}^n
|u^j_x| w_i^2 dx
+ O(1) \delta^{{1}/{2}} \int_\mathbb{R}\sum_{i=1 \atop i \neq p}^n |u^i_x| w_i^2 dx\\
&+O(1) \delta^{-{1}/2} \int_\mathbb{R} |u^p_x|^2 w_p^2 dx
+ O(1) \big(\delta^{{1}/{4}} + \varepsilon_0\big)
\big(\|W_x\|^2 + \|W_t\|^2\big)\\
&+ O(1) \bigg[\int_\mathbb{R} e |W| dx
+\varepsilon_0 \delta (1+t)^{-{5}/{4}}\bigg].
\end{split}
\end{equation}
Here the error term $e$, the vector $B$ and $\rm{I}_2$ are given by $(\ref{ee})$, \eqref{D}
and \eqref{i1i2} respectively.
\end{lemma}
\begin{proof}
From the definition of $B$ in \eqref{D} and of $\rm{I}_2$ in \eqref{i1i2}, we have
\begin{equation}\label{i2sum}
\begin{split}
{\rm{I}}_2 =
\int_\mathbb{R}  W^T \alpha
L(u^a) \big[2 a^2 R(u^a)_x W_x - E_1 - E_2 + Q - 2 R(u^a)_t W_t\big]dx.
\end{split}
\end{equation}
Next, we will estimate the integrand in \eqref{i2sum} as follows. First, by using the Cauchy-Schwartz inequality and  \eqref{uax2}, we can get
\begin{equation*}\label{i2x}
\begin{split}
W^T \alpha L(u^a) R(u^a)_x W_x
& \leqslant O(1) |W||u^a_x||W_x|
\leqslant O(1)\big( \delta^{-{1}/{2}} |W|^2 |u^a_x|^2
+ \delta^{{1}/{2}} |W_x|^2\big) \\
&\leqslant O(1)  \bigg(
\sum_{i,j = 1\atop j \neq i}^n |u^j_x| w_i^2
+\delta^{{1}/{2}}\sum_{i=1 \atop i \neq p}^n |u^i_x| w_i^2
+ \delta^{-{1}/2}  |u^p_x|^2 w_p^2
+ \delta^{{1}/{2}}|W_x|^2\bigg).
\end{split}
\end{equation*}
Then with a similar argument it comes to
\begin{equation*}
W^T \alpha L(u^a) R(u^a)_t W_t
\leqslant O(1) \bigg(
\sum_{i,j = 1\atop j \neq i}^n |u^j_x| w_i^2
+ \delta^{{1}/{2}} \sum_{i=1 \atop i \neq p}^n |u^i_x| w_i^2
+ \delta^{-{1}/2}  |u^p_x|^2 w_p^2
+ \delta^{{1}/{2}}|W_t|^2 \bigg).
\end{equation*}
Here we have used $|u^a_t|\leqslant O(1) |u^a_x|$ again.
According to \eqref{Q} and \eqref{WPhi}, we can get
\begin{equation}\label{Qinequ}
|Q| \leqslant O(1) |\Phi_x|^2
= O(1)|R(u^a)_x W + R(u^a) W_x|^2
\leqslant O(1) \big(|u^a_x W|^2 + |W_x|^2 \big),
\end{equation}
then, the fact mentioned in the previous lemma and \eqref{uax2} yields
\begin{align*}
W^T \alpha L(u^a)Q
\leqslant O(1) \bigg(
\sum_{i,j = 1\atop j \neq i}^n |u^j_x| w_i^2
+ \delta^{{1}/{2}} \sum_{i=1 \atop i \neq p}^n |u^i_x| w_i^2
+ \delta^{-{1}/2}  |u^p_x|^2 w_p^2
+ \varepsilon_0 |W_x|^2\bigg).
\end{align*}
From \eqref{E1}, we know that $E_1$ is controlled
by the error term $e$ which will be estimated later. For the term $E_2$, the detailed computation in \eqref{E2x} gives
\begin{equation*}
E_2 = a^2 u^p_{xt}
= a^2 \big(u^p_x\big)_t
=a^2 r_p(u^p)\rho_{xt}.
\end{equation*}
Then, the a priori assumption \eqref{pa} and the H$\ddot{o}$lder inequality yields
\begin{align*}
- \int_\mathbb{R} W^T \alpha L(u^a) E_2dx
&= - a^2 \int_\mathbb{R} W^T \alpha L(u^a)  r_p(u^p)\rho_{xt}
\leqslant O(1)  \int_\mathbb{R}|W||\rho_{xt}| dx\\
&\leqslant O(1) \|W\| \|\rho_{xt}\|
\leqslant O(1)\varepsilon_0 \delta (1+t)^{-{5}/{4}}.
\end{align*}
Thus, we can obtain the estimate \eqref{I2} by the above inequalities.
\end{proof}

By using the fact introduced in Lemma \ref{lemI1} on the term involving $e$, it concludes from Lemma \ref{lemJ} to Lemma \ref{lemI2} that $(\ref{djgj1})$
becomes the following inequality which is stated in the next proposition.
\begin{proposition}\label{proploestep1}
Under the same assumptions as in
Theorem $\ref{mainthm}$, it holds that
\begin{equation}\label{ji1i2}
\begin{split}
&\frac{1}{2}\frac{d}{dt} \int_\mathbb{R}  \sum_{i=1}^n
\big(
\alpha_i w_i^2 +2 \alpha_i w_i w_{it}
\big) dx
+O(1) m \int_\mathbb{R}\sum_{i,j = 1 \atop j \neq i}^n |u^j_x| w_i^2 dx
+ a^2 \int_\mathbb{R} \sum_{i=1}^n \alpha_i   w_{ix}^2 dx \\
&\quad+ \frac{1}{4} \int_\mathbb{R} \sum_{i=1 \atop i \neq p}^n\alpha_i |\lambda_i(u^i)_x| w_i^2 dx
- \int_\mathbb{R} \sum_{i=1}^n \alpha_i  w_{it}^2dx
\leqslant
O(1) \delta^{-{1}/{2}} \int_\mathbb{R} |u^p_x|^2 w_p^2 dx\\
&\quad+ O(1)\big(\delta^{{1}/{4}} + \varepsilon_0\big ) \big(\|W_x\|^2 + \|W_t\|^2\big)
+ O(1) \bigg[\int_\mathbb{R} e |W|dx
+ \varepsilon_0 \delta (1+t)^{-{5}/{4}}\bigg].
\end{split}
\end{equation}
Here the error term $e$, $\varepsilon_0$ and $\alpha_i$ are given by
\eqref{ee}, \eqref{pa} and \eqref{alphai} respectively.
\end{proposition}

\subsubsection{Estimate for the negative term}
\quad ~
To control the negative term in \eqref{ji1i2}, we multiply $(\ref{lo6})$ by $2W_t^T$ and integrate on $\mathbb{R}$ with respect to $x$ to obtain
\begin{equation}\label{los1}
\begin{split}
&\frac{d}{dt}  \int_\mathbb{R}  \sum_{i=1}^n
\big(a^2 w_{ix}^2 + w_{it}^2\big) dx
+2 \int_\mathbb{R} \sum_{i=1}^n \big[w_{it}^2
+\lambda_i(u^a) w_{it} w_{ix}\big] dx
= \rm{I}_3 + \rm{I}_4,
\end{split}
\end{equation}
where
\begin{equation}\label{i3i4}
 {\rm{I}_3} =  - 2\int_\mathbb{R}  W_t^T A W dx,\quad\quad
 {\rm{I}_4} = 2\int_\mathbb{R}  W_t^T B dx,
\end{equation}
where $A$ and $B$ are defined by \eqref{B} and \eqref{D}.
Next, we will derive the estimate for $\rm{I}_3$ and $\rm{I}_4$, which is
similar to the argument on $\rm{I}_1$ and $\rm{I}_2$. From the definition
of $A$ in \eqref{B} and of $\rm{I}_3$ in \eqref{i3i4}, we have
\begin{equation}\label{i3sum}
{\rm{I}}_3 =
- 2\int_\mathbb{R} \sum_{i,j=1}^n w_{it} a_{ij} w_j dx
\leqslant  O(1)\int_\mathbb{R} \sum_{i,j=1}^n |a_{ij}|| w_{it} w_j| dx.
\end{equation}
From \eqref{bij3}, we can get
\begin{equation*}
\begin{split}
|a_{ij}|\leqslant O(1)\bigg( \sum_{k=1}^n |u^k_x| + |\rho_{xxt}|+ e\bigg).
\end{split}
\end{equation*}
Then \eqref{i3sum} comes to
\begin{equation}\label{i3bij}
{\rm{I}}_3
\leqslant O(1) \int_\mathbb{R} \sum_{i,j=1}^n
\bigg( \sum_{k=1}^n |u^k_x| + |\rho_{xxt}|+ e\bigg)
| w_{it} w_j| dx.
\end{equation}
Next, we will estimate the terms of the integrand in \eqref{i3bij}.
\begin{equation*}
\begin{split}
\sum_{i,j,k=1}^n |u^k_x|| w_{it} w_j|
\leqslant& O(1)\delta^{-1/2}\sum_{j,k=1}^n |u^k_x|^2| w_j|^2 + O(1)\delta^{{1}/{2}} |W_t|^2\\
\leqslant& O(1)\delta^{1/2}\bigg(
\sum_{i,j = 1 \atop j \neq i}^n|u^j_x| w_i^2
+ \sum_{i=1 \atop i \neq p }^n |u^i_x| w_i^2
+ \delta^{-1}|u^p_x|^2 w_p^2\bigg)+ O(1)\delta^{{1}/{4}} |W_t|^2.
\end{split}
\end{equation*}
Similar to Lemma \ref{lemI1}, we can get the fact $\|W_t\|_{L^\infty} \leqslant \varepsilon_0$, then the estimate about the term involving
$\rho_{xxt}$ in \eqref{i3bij} is the same as \eqref{rhoxxt}. Using this fact on the term
about $e$, we obtain
\begin{equation}\label{I3}
\begin{split}
\rm{I}_3
&\leqslant  O(1) \delta^{{1}/{2}}\int_\mathbb{R} \bigg( \sum_{i,j = 1 \atop j \neq i}^n
|u^j_x| w_i^2
+ \sum_{i=1 \atop i \neq p }^n |u^i_x| w_i^2 \bigg)dx
+ O(1) \delta^{-{1}/2} \int_\mathbb{R} |u^p_x|^2 w_p^2 dx
\\
&\quad
+ O(1)\delta^{{1}/{4}}\|W_t\|^2
+ O(1) \bigg[\int_\mathbb{R} e|W| dx + \varepsilon_0 \delta (1+t)^{-3/2} \bigg].
\end{split}
\end{equation}
Now we work on $\rm{I}_4$. From the definition
of $B$ in \eqref{D} and of $\rm{I}_4$ in \eqref{i3i4}, we have
\begin{equation}\label{I4}
\begin{split}
{\rm{I}_4} =& 2\int_\mathbb{R}  W_t^T B dx
=2\int_\mathbb{R}  W_t^T L(u^a)
\big[2 a^2 R(u^a)_x W_x - E_1 - E_2 + Q - 2 R(u^a)_t W_t\big]dx\\
\leqslant&
O(1) \delta^{{1}/{2}}\int_\mathbb{R}\bigg( \sum_{i,j = 1\atop j \neq i}^n
|u^j_x| w_i^2
+ \sum_{i=1 \atop i \neq p}^n |u^i_x| w_i^2 \bigg)dx
+O(1) \delta^{-{1}/2} \int_\mathbb{R} |u^p_x|^2 w_p^2 dx\\
&+ O(1) \big(\delta^{{1}/{4}} + \varepsilon_0\big) \big(\|W_x\|^2 + \|W_t\|^2\big)
+ O(1) \bigg[\int_\mathbb{R} e |W_t| dx
+\varepsilon_0 \delta (1+t)^{-{5}/{4}}\bigg].
\end{split}
\end{equation}
Here we have used
\begin{equation*}
\begin{split}
W_t^T L(u^a) \big[2 a^2 R(u^a)_x W_x + Q - 2R(u^a)_t W_t \big]
\leqslant O(1) (\delta + \varepsilon_0 )(|W_t|^2 + |W_x|^2 +|u^a_x|^2 |W|^2),
\end{split}
\end{equation*}
which is based on the Cauchy-Schwartz inequality and the fact
$\|W_t\|_{L^\infty} \leqslant \varepsilon_0$.
The last inequality together with \eqref{uax2} and the H$\ddot{o}$lder inequality
gives \eqref{I4}.
It concludes from \eqref{I3} and \eqref{I4} that \eqref{los1} becomes
the following inequality which is stated in the next proposition.
\begin{proposition}\label{proploestep2}
Under the same assumptions as in
Theorem $\ref{mainthm}$, it holds that
\begin{equation}\label{i3i4sum}
\begin{split}
&\frac{d}{dt}  \int_\mathbb{R}  \sum_{i=1}^n
\big(a^2 w_{ix}^2 + w_{it}^2\big) dx
+2 \int_\mathbb{R} \sum_{i=1}^n \big[w_{it}^2
+ \lambda_i(u^a) w_{it} w_{ix}\big] dx
\leqslant
O(1)\varepsilon_0 \delta (1+t)^{-{5}/{4}}\\
&\quad+O(1) \delta^{{1}/{2}} \int_\mathbb{R} \bigg(\sum_{i,j = 1\atop j \neq i}^n|u^j_x| w_i^2
+ \sum_{i=1 \atop i \neq p}^n |u^i_x| w_i^2 \bigg)dx
+O(1) \delta^{-{1}/2} \int_\mathbb{R} |u^p_x|^2 w_p^2dx\\
&\quad + O(1)\int_\mathbb{R} e (|W|+|W_t|) dx
+  O(1)\big(\delta^{{1}/{4}} + \varepsilon_0\big) \big(\|W_x\|^2 + \|W_t\|^2\big).
\end{split}
\end{equation}
Here the error term $e$ and $\varepsilon_0$ are given by
\eqref{ee} and \eqref{pa} respectively.
\end{proposition}

Thus, adding \eqref{i3i4sum} to \eqref{ji1i2} yields
\begin{equation*}
\begin{split}
&\frac{1}{2}\frac{d}{dt} \int_\mathbb{R}  \sum_{i=1}^n
\Big[\big(
\alpha_i w_i^2 +2 \alpha_i w_i w_{it} + 2  w_{it}^2
\big)+2a^2 w_{ix}^2\Big] dx
+\frac{1}{8} \int_\mathbb{R}\bigg(\sum_{i,j = 1 \atop j \neq i}^n  |u^j_x| w_i^2
+ \sum_{i=1 \atop i \neq p}^n |u^i_x| w_i^2\bigg) dx\\
&~+  \int_\mathbb{R}  \sum_{i=1}^n \Big[
a^2 \alpha_i w_{ix}^2 + 2 \lambda_i(u^a) w_{it} w_{ix} + (2-\alpha_i)w_{it}^2\Big]dx
\leqslant O(1)\big(\delta^{{1}/{4}} + \varepsilon_0\big) \big(\|W_x\|^2 + \|W_t\|^2\big)\\
&~ + O(1)\delta^{-{1}/{2}} \int_\mathbb{R} |u^p_x|^2 w_p^2 dx
+ O(1) \bigg[\int_\mathbb{R} e (|W|+|W_t|) dx
+\varepsilon_0 \delta (1+t)^{-{5}/{4}}\bigg].
\end{split}
\end{equation*}
Because we can get the following two facts based on the Cauchy-Schwartz inequality,
the sub-characteristic condition \eqref{scc} and \eqref{alphaiproperty}
\begin{equation}\label{addineq1}
\alpha_i w_i^2 +2 \alpha_i w_i w_{it} + 2  w_{it}^2
\geqslant O(1) \big(w_i^2 + w_{it}^2\big),
\end{equation}
\begin{equation}\label{los3}
\begin{split}
a^2  \alpha_i w_{ix}^2
+   (2 - \alpha_i) w_{it}^2
+ 2 \lambda_i(u^a) w_{it} w_{ix}
\geqslant O(1)  \big(w_{ix}^2 +w_{it}^2\big) .
\end{split}
\end{equation}
So, on the basis of Proposition \ref{proploestep1} and Proposition \ref{proploestep2},
 we obtain
\begin{equation}\label{los9}
\begin{split}
&\frac{d}{dt} \big( \|W\|_1^2 + \|W_t\|^2\big)
+\frac{1}{8} \int_\mathbb{R}\bigg(\sum_{i,j = 1 \atop j \neq i}^n  |u^j_x| w_i^2
+ \sum_{i=1 \atop i \neq p}^n |u^i_x| w_i^2\bigg) dx
+\frac{1}{8} \big(\|W_x\|^2 + \|W_t\|^2\big)\\
&\leqslant O(1)\delta^{-{1}/{2}} \int_\mathbb{R} |u^p_x|^2 w_p^2 dx
+ O(1) \bigg[\int_\mathbb{R} e (|W|+|W_t|) dx
+\varepsilon_0 \delta (1+t)^{-{5}/{4}}\bigg].
\end{split}
\end{equation}
Integrating $(\ref{los9})$ from 0 to $t$ with respect to the temporal variable, one can get
\begin{equation}\label{los10}
\begin{split}
&\|W(t)\|_1^2 + \|W_t(t)\|^2
+\frac{1}{8}\int_0^t
\int_\mathbb{R}\bigg(\sum_{i,j = 1 \atop j \neq i}^n |u^j_x| w_i^2
+ \sum_{i=1 \atop i \neq p}^n |u^i_x| w_i^2\bigg) dxd\tau\\
&\quad+ \frac{1}{8} \int_0^t\big(\|W_x\|^2 +  \|W_\tau\|^2\big)d\tau
\leqslant  \|W(0)\|_1^2 + \|W_t(0)\|^2 + O(1) \varepsilon_0 \delta^{{1}/{2}}\\
&\quad+O(1) \delta^{-{1}/{2}} \int_0^t\int_\mathbb{R} |u^p_x|^2 w_p^2 dxd\tau
+ O(1) \int_0^t\int_\mathbb{R} e (|W| + |W_\tau|) dxd\tau.
\end{split}
\end{equation}
Here and thereafter $\|f(t)\|=\|f(\cdot,t)\|$ for any $t\geqslant 0$.
By the definition \eqref{ee} of $e$, the H$\ddot{o}$lder inequality and the a priori assumption \eqref{pa}, we can deal with the first term about $e$ as follows:
\begin{equation}\label{ew}
\begin{split}
&\int_0^t \int_\mathbb{R} e |W| dxd\tau
= O(1) \delta^2 \int_{t_0}^t \int_\mathbb{R}
\exp\big\{-c_0 \delta (\tau+|x|)\big\}|W| dxd\tau\\
&~ + O(1) \delta^2 \int_0^{t_0}  \int_\mathbb{R}
\bigg(
\exp\bigg\{-\frac{(x-x_p)^2}{8(1+\tau)}\bigg\}
+\sum\limits_{i=1\atop i\neq p}^n \exp\big\{-C\delta |x-s_i\tau-x_i|\big\}
\bigg)|W| dxd\tau\\
&\leqslant O(1) \delta^{{3}/{2}}
\int_{t_0}^t\exp\{-c_0\delta \tau\} \|W\| d\tau
+O(1)\varepsilon_0 \delta^2
\int_0^{t_0} \big[(1+\tau)^{{1}/{2}} + \delta^{-1}\big]d\tau\\
&\leqslant O(1)\varepsilon_0 \delta^{{1}/{2}}
+O(1) \varepsilon_0 \Big(\delta^2 t_0^{{3}/{2}} + \delta t_0\Big)
\leqslant O(1) \varepsilon_0 \delta_0^{{1}/{2}},
\end{split}
\end{equation}
where we have used the following fact in the last inequality
\begin{equation*}
t_0
\leqslant O(1) \|u(x,0)-\bar{u}(x,0)\|_{L^1}
\leqslant O(1)\delta_0^{-{1}/{2}},
\end{equation*}
this fact can be derived from the definition \eqref{t0} of $t_0$, \eqref{uubarinitial}
and \eqref{ii1}. The estimate on the first term about $e$ is exactly the same as \eqref{ew}, namely
\begin{equation}\label{ewt}
\int_0^t \int_\mathbb{R} e |W_\tau| dxd\tau
\leqslant O(1) \varepsilon_0 \delta_0^{{1}/{2}}.
\end{equation}
Substituting \eqref{ew} and \eqref{ewt} into \eqref{los10}, one can get
the following inequality which is stated in the next proposition.
\begin{proposition}\label{proploestep22}
Under the same assumptions as in
Theorem $\ref{mainthm}$, it holds that
\begin{equation}\label{loemissp}
\begin{split}
&\|W(t)\|_1^2 + \|W_t(t)\|^2
+\frac{1}{8}\int_0^t
\int_\mathbb{R}\bigg(\sum_{i,j = 1 \atop j \neq i}^n |u^j_x| w_i^2
+ \sum_{i=1 \atop i \neq p}^n |u^i_x| w_i^2\bigg) dxd\tau\\
&\quad+ \frac{1}{8} \int_0^t\big(\|W_x\|^2 +  \|W_\tau\|^2\big)d\tau
\leqslant  \|W(0)\|_1^2 + \|W_t(0)\|^2 + O(1) \varepsilon_0 \delta_0^{{1}/{2}}\\
&\quad+O(1) \delta^{-{1}/{2}} \int_0^t\int_\mathbb{R} |u^p_x|^2 w_p^2 dxd\tau.
\end{split}
\end{equation}
where $\varepsilon_0$ is the positive constant same as \eqref{pa}.
\end{proposition}

\subsubsection{Application of the heat kernel estimate}
\quad~
At the end of the lower-order estimate, we will deal with the last term which has not yet been estimated in \eqref{loemissp} by introducing the following lemma.
\begin{lemma}\label{g}
Under the same assumptions as in
Theorem $\ref{mainthm}$, it holds that
\begin{equation*}
\begin{split}
&\int_0^t\int_\mathbb{R} |u^p_x|^2 w_p^2 dxd\tau
\leqslant
O(1) \delta^{{2}} \int_0^t\int_\mathbb{R} \bigg(
\sum_{i,j = 1\atop j \neq i}^n |u^j_x| w_i^2
+  \sum_{i=1 \atop i \neq p }^n |u^i_x| w_i^2\bigg)dxd\tau
+ O(1)\varepsilon_0\delta_0 \\
&+O(1) \delta^{{2}}  \bigg[\|w_p( t)\|^2 + \|w_{pt}(t)\|^2
+ \|w_p(0)\|^2 + \|w_{pt}(0)\|^2+ \int_0^t (\|W_x\|^2 + \|W_\tau\|^2)d\tau\bigg].
\end{split}
\end{equation*}
where $e$ is given by $(\ref{ee})$.
\end{lemma}
\begin{proof}
If we take $\gamma = 1/(4 a^2)$ in \eqref{hkg}, direct calculation shows
\begin{equation}\label{upgxwp}
\int_0^t\int_\mathbb{R} |u^p_x|^2 w_p^2 dxd\tau
 \leqslant O(1)  \delta^2 \int_0^t\int_\mathbb{R} g_x^2 w_p^2 dxd\tau.
\end{equation}
One can get the following inequality by choosing $h=w_p$ in Lemma \ref{gh}
\begin{equation}\label{gg}
\begin{split}
&\int_0^t\int_\mathbb{R} g_x^2 w_p^2 dxd\tau\\
\leqslant& O(1)\bigg(
\|w_p(0)\|^2
+ \int_0^t \| w_{px}\|^2 d\tau
+ \int_0^t <w_{p\tau}, w_p g^2>_{H^{-1}\times H^1}d\tau
\bigg).
\end{split}
\end{equation}
Taking the $p$-th equation in $(\ref{lo6})$, one can get
\begin{equation}\label{los11}
\begin{split}
w_{pt} + \lambda_p(u^a) w_{px} - a^2 w_{pxx} + w_{ptt}
= - \sum_{j = 1}^n a_{pj}w_j
+ b_p,
\end{split}
\end{equation}
where $a_{pj}$ and $b_p$ has been defined in \eqref{B} and \eqref{D}.
Multiplying $(\ref{los11})$ by $w_p g^2$ and integrating over
$[0,t]\times \mathbb{R}$ to get
\begin{equation}\label{duiou}
\begin{split}
\int_0^t <w_{p\tau}, w_p g^2>_{H^{-1}\times H^1}d\tau
=& \int_0^t\int_\mathbb{R}
\big[ a^2 w_{pxx} -  \lambda_p(u^a) w_{px} - w_{p\tau\tau}\big] w_p g^2 dxd\tau\\
&+ \int_0^t\int_\mathbb{R} \Big(b_p - \sum_{j = 1}^n a_{pj}w_j\Big) w_p g^2 dxd\tau.
\end{split}
\end{equation}
Noticing \eqref{gproperty}, then
by the integration by parts and the Cauchy-Schwartz inequality, we can get
\begin{equation}\label{wpxx}
\begin{split}
a^2 \int_0^t\int_\mathbb{R} w_{pxx} w_p g^2 dxd\tau
=&- a^2  \int_0^t\int_\mathbb{R} \big( w_{px}^2 g^2 + 2 g g_x w_p w_{px}\big)dxd\tau\\
\leqslant&   \frac{1}{2}   \int_0^t \int_\mathbb{R} g_x^2 w_p^2 dxd\tau
+O(1) \int_0^t \|w_{px}\|^2d\tau.
\end{split}
\end{equation}
By the integration by parts, one can get
\begin{align*}
\int_0^t\int_\mathbb{R}\lambda_p(u^a) w_{px}w_p g^2 dxd\tau
=\int_0^t\int_\mathbb{R} \bigg[\frac{1}{2} \lambda_p(u^a)_x w_p^2 g^2
+ \lambda_p(u^a) w_p^2 g g_x \bigg]dxd\tau,
\end{align*}
because the $p$-th characteristic field is linearly degenerate and the $p$-th contact wave speed
is zero, then by using \eqref{E1}, we can get
\begin{equation*}
\begin{split}
\lambda_p(u^a)_x w_p^2 g^2
= \sum_{i=1}^n \nabla\lambda_p(u^a) u^i_x w_p^2 g^2
=& \sum_{i=1}^n \Big\{\nabla\lambda_p(u^i)+
\big[\nabla\lambda_p(u^a)-\nabla\lambda_p(u^i)\big]\Big\}
u^i_x w_p^2 g^2\\
\leqslant& O(1)\bigg(\sum_{i = 1 \atop i \neq p}^n |u^i_x| w_p^2 + e w_p^2\bigg),
\end{split}
\end{equation*}
and
\begin{equation*}
\lambda_p(u^a) w_p^2 g g_x = \big[\lambda_p(u^a)-\lambda_p(u^p)\big] w_p^2 g g_x
\leqslant O(1)e w_p^2.
\end{equation*}
So, we have
\begin{equation}\label{lambda}
\begin{split}
\int_0^t\int_\mathbb{R}\lambda_p(u^a) w_{px}w_p g^2 dxd\tau
\leqslant
O(1) \int_0^t \int_\mathbb{R}
\bigg(\sum_{i = 1 \atop i \neq p}^n |u^i_x| w_p^2
+ e w_p^2 \bigg)dxd\tau.
\end{split}
\end{equation}
Direct calculation yields
$$
w_{p\tau\tau} w_p g^2
= \big[ (w_p w_{p\tau})_\tau - w_{p\tau}^2\big] g^2
\leqslant (w_p w_{p\tau})_\tau g^2 + O(1) w_{p\tau}^2,
$$
we also have
\begin{equation*}
\begin{split}
\int_0^t (w_p w_{p\tau})_\tau g^2 d\tau
\leqslant& O(1)\big[(w_pw_{pt})(x,t)
+ (w_pw_{pt})(x,0)\big]
+ O(1)\int_0^t w_pw_{p\tau}gg_\tau d\tau,
\end{split}
\end{equation*}
because
\begin{equation*}
\begin{split}
\int_0^t w_pw_{p\tau}gg_\tau d\tau
\leqslant O(1)\varepsilon_0 \int_0^t \big( w_{p\tau}^2 + g_\tau^2 \big) d\tau
\leqslant O(1)\int_0^t w_{p\tau}^2 d\tau + O(1)\varepsilon_0 \delta^{1/2},
\end{split}
\end{equation*}
thus, we have
\begin{equation}\label{utt}
\begin{split}
& \int_0^t \int_\mathbb{R} w_{p\tau\tau} w_p g^2 dxd\tau\\
\leqslant& O(1) \bigg(\|w_p(t)\|^2 + \|w_{pt}(t)\|^2
+ \|w_p(0)\|^2 + \|w_{pt}(0)\|^2
+  \int_0^t \|w_{p\tau}\|^2 d\tau
+ \varepsilon_0 \delta^{1/2}\bigg).
\end{split}
\end{equation}
On the basis of \eqref{ew}, \eqref{gg} and \eqref{duiou}-\eqref{utt}, we obtain
\begin{equation}\label{gxwp}
\begin{split}
&\int_0^t\int_\mathbb{R} g_x^2 w_p^2 dxd\tau
\leqslant O(1) \int_0^t\int_\mathbb{R}
\sum_{i=1 \atop i \neq p }^n |u^i_x| w_p^2dxd\tau
+ O(1)\int_0^t (\|w_{px}\|^2 + \|w_{p\tau}\|^2)d\tau
\\
&\quad+O(1) \big(\|w_p( t)\|^2 + \|w_{pt}(t)\|^2
+ \|w_p(0)\|^2 + \|w_{pt}(0)\|^2 + \varepsilon_0 \delta^{{1}/{2}}\big)\\
&\quad
+ \int_0^t\int_\mathbb{R} \Big(b_p - \sum_{j = 1}^n a_{pj}w_j\Big) w_p g^2 dxd\tau .
\end{split}
\end{equation}
Due to
\begin{equation*}
\int_\mathbb{R} \Big(b_p - \sum_{j = 1}^n a_{pj}w_j\Big) w_p g^2 dx
\leqslant O(1) ({\rm{I}_1 + \rm{I}_2}).
\end{equation*}
So, the estimates about the two terms involving $a_{pj}$ and $b_p$ is exactly the
same as Lemma \ref{lemI1} and Lemma \ref{lemI2}, thus \eqref{gxwp} becomes
\begin{equation}\label{gxwpi1i2}
\begin{split}
&\int_0^t\int_\mathbb{R} g_x^2 w_p^2 dxd\tau
\leqslant O(1) \delta^{-1/2}\int_0^t\int_\mathbb{R}|u^p_x| w_p^2 dxd\tau
+ O(1)\int_0^t (\|W_{x}\|^2 + \|W_{\tau}\|^2)d\tau\\
&\quad
+ O(1) \int_0^t\int_\mathbb{R}\bigg(
\sum_{i,j=1 \atop j \neq i }^n |u^j_x| w_i^2
+\sum_{i=1 \atop i \neq p }^n |u^i_x| w_i^2\bigg)dxd\tau
+  O(1) \varepsilon_0 \delta_0^{{1}/{2}}\\
&\quad +O(1) \big(\|w_p( t)\|^2 + \|w_{pt}(t)\|^2
+ \|w_p(0)\|^2 + \|w_{pt}(0)\|^2\big).
\end{split}
\end{equation}
In view of \eqref{upgxwp} and \eqref{gxwpi1i2}, we finish the proof of this lemma.
\end{proof}
Finally, with the help of \eqref{loemissp} and Lemma $\ref{g}$,we obtain the
lower-order estimate, which is stated in the following proposition:
\begin{proposition}\label{lemloe}
Under the same assumptions as in Theorem $\ref{mainthm}$, it holds that
\begin{equation}\label{loefinal}
\begin{split}
&\|W(t)\|_1^2 + \|W_t(t)\|^2
+\int_0^t \int_\mathbb{R} \bigg(
\sum_{i,j = 1\atop j \neq i}^n |u^j_x| w_i^2
+\sum_{i=1 \atop i \neq p}^n |u^i_x| w_i^2
+ |u^p_x|^2w_p^2\bigg)dxd\tau\\
&\quad + \int_0^t(\|W_x\|^2 + \|W_\tau\|^2)d\tau
\leqslant O(1)\big( \|W(0)\|_1^2  + \|W_t(0)\|^2\big)
+  O(1) \varepsilon_0 \delta_0^{{1}/{2}},
\end{split}
\end{equation}
where $\varepsilon_0$ is the positive constant same as $(\ref{pa})$.
\end{proposition}

\subsection{Higher-order estimates}\label{sechoe}
\quad ~
In order to obtain the uniform estimates and close the argument with the a
priori assumption $(\ref{pa})$, we need to estimate the higher order derivatives. For
this purpose, we apply spatial derivative in $(\ref{lo6})$ to get
\begin{equation}\label{ho1}
\begin{split}
(W_x)_t + [\Lambda(u^a) W_x]_x - a^2 (W_x)_{xx} + (W_x)_{tt}
=-\big(AW - B\big)_x,
\end{split}
\end{equation}
where $A$ and $B$ is defined in \eqref{B} and \eqref{D}.
Multiplying $(\ref{ho1})$ by $W_x^T$ and integrate on $\mathbb{R}$
with respect to $x$, by using the Cauchy-Schwartz inequality, we have
\begin{equation*}
\begin{split}
&\frac{1}{2}\frac{d}{dt} \int_\mathbb{R} \big(|W_x|^2 + 2W_{x} W_{xt} \big) dx
+ a^2 \|W_{xx}\|^2 - \|W_{xt}\|^2\\
=&  \int_\mathbb{R} W_{xx}^T \big[\Lambda(u^a) W_x + AW - B\big] dx\\
\leqslant& O(1) \delta^{-1/4} \big(\|W_x\|^2 + \|AW\|^2 + \|B\|^2\big)
+ O(1)\delta^{1/4} \|W_{xx}\|^2.
\end{split}
\end{equation*}
By Lemma \ref{lemmacontact}, Lemma \ref{lemshock}, \eqref{E1E2}, \eqref{E1}
and \eqref{Qinequ}, a direct calculation on $A$ and $B$ gives
\begin{equation}\label{absoluteB}
\begin{split}
|A|
\leqslant O(1) \big(|u^a_x|+|u^a_t|+|u^a_x|^2+|u^a_t|^2+|u^a_{xx}|+|u^a_{tt}|\big)
\leqslant O(1) \big(|u^a_x|+|u^p_{xx}|+|u^p_{tt}|\big),
\end{split}
\end{equation}
\begin{equation*}
\begin{split}
|B|
\leqslant
O(1) \big(|u^a_xW_x|+e+|u^p_{xt}|+|u^a_x W|^2+|W_x|^2+|u^a_tW_t| \big).
\end{split}
\end{equation*}
Here we also used $|u^a_t|\leqslant O(1) |u^a_x|$.
By the a priori assumption \eqref{pa}, we can get the fact
$\|W_x\|_{L^\infty}\leqslant O(1)\varepsilon_0$. Then,
using the above two inequalities, this fact,
we have
\begin{equation*}
\|AW\|^2 + \|B\|^2
\leqslant
O(1) \int_\mathbb{R} \big(|u^a_x W|^2 +|u^p_{xx}|^2+|u^p_{tt}|^2+ e^2 +|u^p_{xt}|^2 \big)dx
+O(1) \big(\|W_x\|^2 + \|W_t\|^2\big).
\end{equation*}
So far, the following
inequality is what we get
\begin{equation}\label{ttt}
\begin{split}
&\frac{1}{2}\frac{d}{dt} \int_\mathbb{R} \big(|W_x|^2 + 2W_{x} W_{xt} \big) dx
+ a^2\|W_{xx}\|^2 - \|W_{xt}\|^2
\leqslant O(1) \delta^{-1/4}  \big(\|W_x\|^2 + \|W_t\|^2\big)\\
&\quad+O(1) \delta^{-1/4}  \int_\mathbb{R} \big(|u^a_x W|^2 +|u^p_{xx}|^2+|u^p_{tt}|^2+ e^2 +|u^p_{xt}|^2\big)dx
+ O(1)\delta^{1/4} \|W_{xx}\|^2.
\end{split}
\end{equation}
Integrating \eqref{ttt} on $[0,t]$ with respect to the temporal variable, we have
\begin{equation}\label{ho3}
\begin{split}
&\frac{1}{2}\int_\mathbb{R}\big(|W_x|^2 +  2W_{x} W_{xt}\big)(x,t) dx
+ a^2\int_0^t \|W_{xx}\|^2d\tau - \int_0^t \|W_{x\tau}\|^2d\tau\\
\leqslant&  O(1) \delta^{-1/4} ( \|W(0)\|_1^2 + \|W_t(0)\|_1^2)
+  O(1) \delta_0^{{1}/{4}}+ O(1)\delta^{1/4}\int_0^t \|W_{xx}\|^2d\tau,
\end{split}
\end{equation}
where we have use Proposition \ref{lemloe}, \eqref{ew} and the following inequality
\begin{equation*}
\begin{split}
\int_0^t  \int_\mathbb{R} \big(|u^p_{xx}|^2+|u^p_{\tau\tau}|^2+|u^p_{x\tau}|^2 \big)dxd\tau
\leqslant& O(1)\delta^{1/2}.
\end{split}
\end{equation*}

To control the negative term in $(\ref{ho3})$, multiplying $(\ref{ho1})$ by $2 W_{xt}^T$ and integrate on $\mathbb{R}$ with respect to $x$, one can get
\begin{equation}\label{negative}
\begin{split}
&\frac{d}{dt}\int_\mathbb{R} \big(a^2 |W_{xx}|^2 + |W_{xt}|^2 \big)dx
+  2\int_\mathbb{R}\big[|W_{xt}|^2 +  W_{xt}^T \Lambda(u^a) W_{xx} \big]dx\\
=& 2 \int_\mathbb{R}  W_{xt}^T (-AW + B)_x dx
- 2 \int_\mathbb{R}W_{xt}^T \Lambda(u^a)_x W_x dx\\
\leqslant& O(1) \delta^{-1/4} \big(\|A_xW\|^2 +\|B_x\|^2+\|W_x\|^2 \big)
+ O(1) \delta^{1/4} \|W_{xt}\|^2.
\end{split}
\end{equation}
Here we have used \eqref{absoluteB}. Since, by using the a priori assumption, Lemma \ref{lemmacontact} and Lemma \ref{lemshock}, we also have
\begin{equation*}
|A_x|\leqslant O(1) \big(|u^a_x| + |u^p_{xx}| + |u^p_{tt}|
+ |u^p_{xt}| + |u^p_{xxx}| + |u^p_{xtt}|\big),
\end{equation*}
\begin{equation*}
|B_x|\leqslant O(1) \big[|W_x| + |W_t| + e + |u^p_{xt}| + |u^p_{xxt}|
+ |u^a_xW|^2   + (\delta+\varepsilon_0)(|W_{xx}|+ |W_{xt}|)\big].
\end{equation*}
Thus, integrating \eqref{negative} over $[0,t]$ with respect to the temporal variable,
in view of Proposition \ref{lemloe}, we can similarly get
\begin{equation*}\label{ho4}
\begin{split}
& \int_\mathbb{R} \big(a^2 |W_{xx}|^2 + |W_{xt}|^2\big)(x,t)dx
+2\int_0^t\int_\mathbb{R} \big[|W_{x\tau}|^2
+ W_{x\tau}^T\Lambda(u^a)W_{xx}\big]dxd\tau\\
\leqslant&  O(1)\delta^{-1/4}\big( \|W(0)\|_2^2 + \|W_t(0)\|_1^2\big)
+ O(1) \delta_0^{{1}/{4}}
+O(1) (\delta^{1/4}+\varepsilon_0)\int_0^t\big(\|W_{xx}\|^2+\|W_{x\tau}\|^2\big)d\tau.
\end{split}
\end{equation*}
Adding the last inequality to \eqref{ho3}, one can get
\begin{equation*}
\begin{split}
&\frac{1}{2}\int_\mathbb{R}\Big[\big(|W_x|^2 +  2W_{x} W_{xt}+ 2|W_{xt}|^2\big)
+ 2 a^2 |W_{xx}|^2 \Big](x,t) dx\\
&\quad +\int_0^t\int_\mathbb{R} \big[|W_{x\tau}|^2
+ 2W_{x\tau}^T\Lambda(u^a)W_{xx}+ a^2|W_{xx}|^2\big]dxd\tau\\
\leqslant&  O(1)\delta^{-1/4}( \|W(0)\|_1^2 + \|W_t(0)\|_1^2)
+  O(1) \delta_0^{{1}/{4}}\\
&\quad+O(1) (\delta^{1/4}+\varepsilon_0)\int_0^t\big(\|W_{xx}\|^2+\|W_{x\tau}\|^2\big)d\tau.
\end{split}
\end{equation*}
By using the Cauchy-Schwartz inequality and the sub-characteristic condition \eqref{sc},
we can get the following two inequalities
$$
|W_x|^2 +  2W_{x} W_{xt}+ 2|W_{xt}|^2 \geqslant O(1)\big(|W_x|^2 + |W_{xt}|^2 \big),
$$
$$
|W_{x\tau}|^2+ 2W_{x\tau}^T\Lambda(u^a)W_{xx}+ a^2|W_{xx}|^2
\geqslant O(1)\big(|W_{xx}|^2 + |W_{x\tau}|^2 \big).
$$
Thus, we obtain the second order derivative estimate which is stated in the
following proposition
\begin{proposition}
Under the same assumptions as in Theorem $\ref{mainthm}$, it holds that
\begin{equation}\label{ho5}
\begin{split}
&\|W_x(t)\|_1^2 + \|W_{xt}(t)\|^2
+ \int_0^t \big(\|W_{xx}\|^2 + \|W_{x\tau}\|^2\big) d\tau\\
 \leqslant&  O(1)\delta^{-1/4}\big( \|W(0)\|_2^2 + \|W_t(0)\|_1^2\big)
+  O(1) \delta_0^{{1}/{4}}.
\end{split}
\end{equation}
Here $\delta_0$ is a small positive constant same as \eqref{ii1}.
\end{proposition}

Based on the derivative in the system $(\ref{ho1})$ with respect to $x$ and do the above estimate again through multiplying the derivative equation by $W_{xx}^T$ and $ 2W_{xxt}^T$, we can get the third order derivative estimate as follows
\begin{equation}\label{ho6}
\begin{split}
&\|W_{xx}(t)\|_1^2  + \|W_{xxt}( t)\|^2
+ \int_0^t \big(\|W_{xxx}\|^2 + \|W_{xx\tau}\|^2\big) d\tau\\
\leqslant&  O(1)\delta^{-1/4}\big( \|W(0)\|_3^2 + \|W_t(0)\|_2^2\big)
+  O(1) \delta_0^{{1}/{4}}.
\end{split}
\end{equation}

Combining the lower order estimate $(\ref{loefinal})$ and the higher order estimate $(\ref{ho5})$
and $(\ref{ho6})$, we obtain the following desired energy estimate
\begin{align*}
&\|W(t)\|_3^2 + \|W_t(t)\|_2^2
+ \int_0^t  \big(\|W_x\|_2^2 + \|W_\tau\|_2^2\big) d\tau\\
 \leqslant&  O(1)\delta^{-1/4}\big( \|W(0)\|_3^2 + \|W_t(0)\|_2^2
\big)
+  C_3 \delta_0^{{1}/{4}}.
\end{align*}
Since $\Phi(x,t) = R(u^a) W(x,t)$, we have
\begin{equation}\label{de}
\begin{split}
&\|\Phi( t)\|_3^2 + \|\Phi_t( t)\|_2^2
+ \int_0^t  (\|\Phi_x\|_2^2 + \|\Phi_\tau\|_2^2) d\tau\\
 \leqslant&  O(1)\delta^{-1/4}( \|\Phi(0)\|_3^2 + \|\Phi_t(0)\|_2^2)
+  C_3 \delta_0^{{1}/{4}}.
\end{split}
\end{equation}

\subsection{Proof of the main theorem}\label{secmainthm}
\quad ~
The a priori assumption \eqref{pa}will be closed if we choose $\delta_0$ so small in
\eqref{de} such that
\begin{equation*}
2C_3 \delta_0^{{1}/{4}} \leqslant \varepsilon_0.
\end{equation*}
It should be noted that the smallness assumptions \eqref{ii1} imply the
smallness of $\delta$ and $\varepsilon_0$.

Now that the a priori assumption \eqref{pa} has been proved,
combining this with the local existence theorem,
we can prove that the initial value problem
$(\ref{lp1.1-1})$ with \eqref{lp1.1-3} has a global solution
satisfying
\begin{equation*}
u(x,t)\in C([0, +\infty); H^2) \cap L^2(0, +\infty; H^3),
\end{equation*}
\begin{equation*}
v(x,t)\in C([0, +\infty); H^1) \cap L^2(0, +\infty; H^2),
\end{equation*}
for small $\delta_0$, by a standard continuation argument.
Furthermore, $(\ref{de})$ imply that
\begin{equation}
\|(u-u^a, v-v^a)(\cdot, t)\|_{L^\infty}
= \|(\Phi_x, \Phi_t)(\cdot, t)\|_{L^\infty}
\rightarrow 0, \quad \mbox{as} ~ t \rightarrow +\infty.
\end{equation}
Thus, Theorem \ref{mainthm} has been proved.

\noindent {\bf Acknowledgements:}
The author is grateful to Professor Tao Luo for his advices and suggestions.
The author wishes to sincerely thank Professor Huihui Zeng for many very helpful discussions,
sharing with me her insights and guidances.
This research was partially supported by NSFC grant \#11301293/A010801.


\end{document}